\documentclass[aap]{imsart}

\RequirePackage{amsthm,amsmath,amsfonts,amssymb}
\RequirePackage[numbers]{natbib}
\RequirePackage[colorlinks,citecolor=blue,urlcolor=blue]{hyperref}
\RequirePackage{graphicx}

\startlocaldefs
\newtheorem{theorem}{Theorem}
\newtheorem{lemma}{Lemma}
\newtheorem{corollary}{Corollary}
\newtheorem{proposition}{Proposition}
\newtheorem{condition}{Condition}
\usepackage{mathtools}
\mathtoolsset{showonlyrefs}
\renewcommand{\P}{\mathbb{P}} 
\newcommand{\E}{\mathbb{E}} 
\newcommand{\R}{\mathbb{R}} 
\newcommand{\N}{\mathbb{N}} 
\newcommand{\Z}{\mathbb{Z}} 
\newcommand{\trans}{\mathsf{T}}
\newcommand{\mix}{\text{mix}}
\newcommand{\rel}{\text{rel}}
\newcommand{\hit}{\text{hit}}
\newcommand{\ent}{\text{ent}}
\DeclarePairedDelimiter{\ceil}{\lceil}{\rceil}
\DeclarePairedDelimiter{\floor}{\lfloor}{\rfloor}
\renewcommand{\epsilon}{\varepsilon}
\usepackage[shortlabels]{enumitem}
\endlocaldefs

\begin{document}

\begin{frontmatter}
\title{Restart perturbations for reversible Markov chains: trichotomy and pre-cutoff equivalence}
\runtitle{Restart perturbations for reversible Markov chains}

\begin{aug}
\author[A]{\fnms{Daniel} \snm{Vial}\ead[label=e1]{dvial@utexas.edu}} \and
\author[B]{\fnms{Vijay} \snm{Subramanian}\ead[label=e2]{vgsubram@umich.edu}}
\address[A]{Electrical and Computer Engineering,
University of Texas at Austin,
\printead{e1}}
\address[B]{Electrical Engineering and Computer Science, University of Michigan,
\printead{e2}}
\end{aug}

\begin{abstract}
Given a reversible Markov chain $P_n$ on $n$ states, and another chain $\tilde{P}_n$ obtained by perturbing each row of $P_n$ by at most $\alpha_n$ in total variation, we study the total variation distance between the two stationary distributions, $\| \pi_n - \tilde{\pi}_n \|$. We show that for chains with \textit{cutoff}, $\| \pi_n - \tilde{\pi}_n \|$ converges to $0$, $e^{-c}$, and $1$, respectively, if the product of $\alpha_n$ and the mixing time of $P_n$ converges to $0$, $c$, and $\infty$, respectively. This echoes recent results for specific random walks that exhibit cutoff, suggesting that cutoff is the key property underlying such results. Moreover, we show $\| \pi_n - \tilde{\pi}_n \|$ is maximized by \textit{restart perturbations}, for which $\tilde{P}_n$ ``restarts'' $P_n$ at a random state with probability $\alpha_n$ at each step. Finally, we show that \textit{pre-cutoff} is (almost) equivalent to a notion of ``sensitivity to restart perturbations,'' suggesting that chains with sharper convergence to stationarity are inherently less robust.
\end{abstract}

\end{frontmatter}

\section{Introduction} \label{secIntro}

Markov chains are common tools for modeling phenomena such as the movement of asset prices in financial markets or the processing of tasks in data centers. A fundamental concern is how modeling inaccuracies affect the chain's steady-state behavior, i.e.\ how changes to the chain's transition matrix affect its stationary distribution. Mathematically, we formalize this as follows. Let $P_n$ be the transition matrix of a Markov chain with $n$ states and stationary distribution $\pi_n$. Denote by $\tilde{P}_n$ the transition matrix and $\tilde{\pi}_n$ the stationary distribution of another chain, obtained by perturbing each row of $P_n$ by at most $\alpha_n \in (0,1)$ (in total variation distance). Then the main question we study is as follows: how does the perturbation magnitude $\alpha_n$ relate to the error magnitude $\| \pi_n - \tilde{\pi}_n \|$ (where $\| \cdot \|$ denotes total variation) as the number of states $n$ grows?

Before previewing our results, we briefly outline two basic notions that play prominent roles. The first notion is a class of perturbations we call \textit{restart perturbations}, for which $\tilde{P}_n$ is obtained from $P_n$ as follows. From the current state, flip a coin that lands heads with probability $\alpha_n$. If heads, sample the next state from some auxiliary distribution $\sigma_n$, i.e.\ ``restart'' the chain at a random state, distributed as $\sigma_n$. If tails, sample the next state from $P_n$, i.e.\ follow the original chain. In the case where $P_n$ describes the simple random walk on some underlying graph, this perturbation is more commonly known as PageRank \cite{page1999pagerank}, a model for Internet browsing (nodes and edges represent web pages and hyperlinks, respectively; choosing the next state from $P_n$ corresponds to following a hyperlink and ``restarting'' corresponds to typing in a new page's web address). Also, this perturbation yields an example of a \textit{Doeblin chain}, for which so-called ``perfect sampling'' is possible \cite{athreya2000perfect,propp1996exact}.

A second important notion is that of mixing times and cutoff. Roughly, the $\epsilon$-\textit{mixing time} $t_{\mix}^{(n)}(\epsilon)$ is the number of steps the chain with transition matrix $P_n$ must take before its distribution is $\epsilon$-close to $\pi_n$ (see \eqref{eqDefnTmix} for a formal definition). Certain chains exhibit \textit{cutoff}, meaning
\begin{equation} \label{eqDefnCutoffIntro}
\lim_{n \rightarrow \infty}  \frac{t_{\mix}^{(n)}(\epsilon)}{ t_{\mix}^{(n)}(1-\epsilon) } = 1\ \forall\ \epsilon \in (0,1/2) .
\end{equation}
A weaker condition is \textit{pre-cutoff}, which only requires
\begin{equation} \label{eqDefnPrecutoffIntro}
\sup_{\epsilon \in (0,1/2)} \limsup_{n \rightarrow \infty} \frac{ t_{\mix}^{(n)}(\epsilon) }{ t_{\mix}^{(n)}(1-\epsilon)} < \infty .
\end{equation}
Intuitively, a chain with cutoff or pre-cutoff has a particularly ``sharp'' convergence to stationarity: its distribution remains $(1-\epsilon)$-far from $\pi_n$ for many steps, then abruptly (i.e.\ on a much shorter timescale) becomes $\epsilon$-close to $\pi_n$.

In this work, we show that a chain's perturbation sensitivity is intimately related to its mixing time and to how sharply it converges to stationarity. To motivate this intuitively, we contrast two example chains in Figure \ref{fig_intro} (see Section \ref{secExamples} for details). At left, we plot the total variation distance between the $t$-step distribution and $\pi_n$, as a function of $t$ (see \eqref{eqDefnDnT} for a formal definition). At right, we plot the perturbation error $\| \pi_n - \tilde{\pi}_n \|$ for particular restart perturbations with restart probability $\alpha_n \propto 1 /\sqrt{ t_{\mix}^{(n)}(\epsilon) }$, as a function of $n$. We contrast the two chains as follows:
\begin{itemize}
\item Chain 1 has an extremely sharp convergence to stationarity (and in fact exhibits a strong form of cutoff -- see Section \ref{secExamplesAnalysis}). Thus, after $\sqrt{ t_{\mix}^{(n)}(\epsilon) }$ steps -- when the first restart occurs, in expectation -- its distribution remains far from $\pi_n$. Consequently, we can choose a restart distribution $\sigma_n$ to move the chain to a part of the state space with low stationary measure. This is turn drives $\tilde{\pi}_n$ far from $\pi_n$.
\item In contrast, Chain 2 has a very gradual convergence to stationarity (which is in fact ``maximally gradual" in a certain sense -- see Section \ref{secExamplesAnalysis}). Consequently, when the first restart occurs, it has made significant progress toward its stationary distribution $\pi_n$ and cannot wander too far from this distribution, despite the restart.
\end{itemize}

\begin{figure}
\centering
\includegraphics[width=\textwidth]{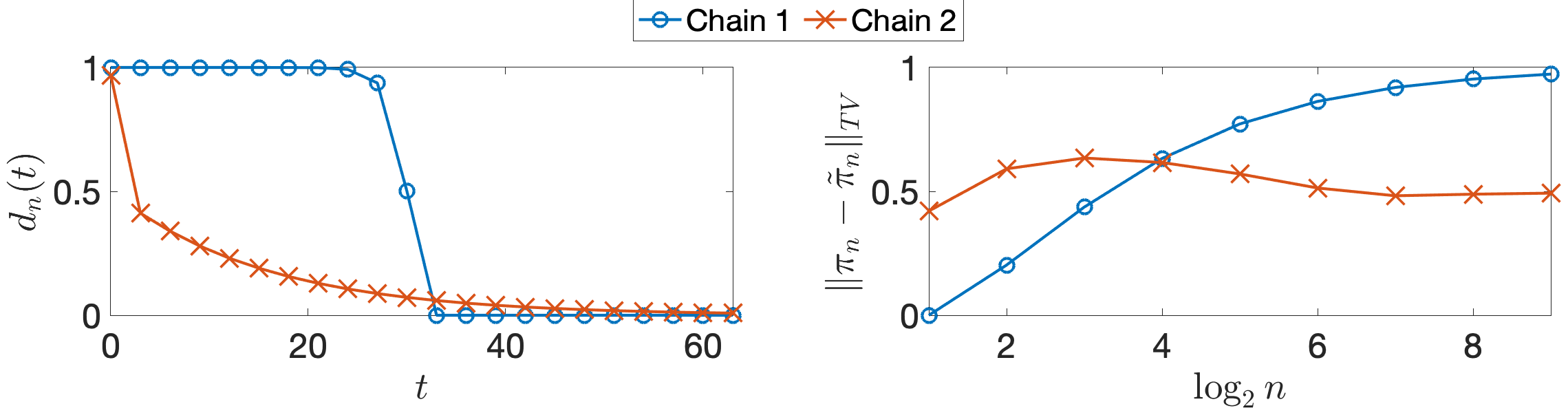}
\caption{Convergence when $n = 32$ (left) and perturbation error (right) for example chains.} \label{fig_intro}
\end{figure}

\subsection{Preview of main results}

The goal of this work is to generalize the insights of Figure \ref{fig_intro} beyond the example chains, and beyond restart perturbations. In particular, we investigate the more general class of $\alpha_n$-bounded perturbations, for which the rows of $P_n$ and $\tilde{P}_n$ each differ by at most $\alpha_n$ in total variation.

Our first main result, Theorem \ref{thmAllRegimes}, says that the relative asymptotics of $\alpha_n$ and $t_{\mix}^{(n)}(\epsilon)$ fully characterize the asymptotics of $\| \pi_n - \tilde{\pi}_n \|$ for a particular class of chains $P_n$. More specifically, we prove that the following trichotomy occurs:
\begin{itemize}
\item If $\lim_{n \rightarrow \infty} \alpha_n t_{\mix}^{(n)}(\epsilon) = 0$, then $\lim_{n \rightarrow \infty} \| \pi_n - \tilde{\pi}_n \| = 0$ for any $\alpha_n$-bounded perturbation, i.e.\ no such perturbation affects the stationary distribution.
\item If $\lim_{n \rightarrow \infty} \alpha_n t_{\mix}^{(n)}(\epsilon) = \infty$, then $\lim_{n \rightarrow \infty} \| \pi_n - \tilde{\pi}_n \| = 1$ for some $\alpha_n$-bounded perturbation, i.e.\ some such perturbation maximally affects the stationary distribution. In particular, we construct a restart perturbation for which $\| \pi_n - \tilde{\pi}_n \| \rightarrow 1$, suggesting that restart perturbations are worst-case among the general class of bounded perturbations.
\item If $\lim_{n \rightarrow \infty} \alpha_n t_{\mix}^{(n)}(\epsilon) = c \in (0,\infty)$, an intermediate behavior occurs: all $\alpha_n$-bounded perturbations satisfy $\limsup_{n \rightarrow \infty} \| \pi_n - \tilde{\pi}_n \| \leq 1 - e^{-c}$, and some $\alpha_n$-bounded perturbation -- again, a restart perturbation -- attains the bound.
\end{itemize}

We note that Theorem \ref{thmAllRegimes} holds assuming the original chain is lazy ($P_n(i,i) \geq 1/2\ \forall\ i$), reversible ($\pi_n(i) P_n(i,j) = \pi_n(j) P_n(j,i)\ \forall\ i,j$), and exhibits cutoff. The laziness and reversibility assumptions are inherited from \cite{basu2015characterization}, which contains an inequality used to prove our lower bounds (see Section \ref{secProofLower}). Hence, we suspect these assumptions may be artifacts of our analysis. In contrast, we believe some notion of cutoff is necessary (as will be discussed shortly). We also note that the proof Theorem \ref{thmAllRegimes} contains several intermediate results that hold under weaker assumptions and may be of independent interest (see Section \ref{secProofTrichotomy}):
\begin{itemize}
\item Most of our upper bounds only require ergodicity. In particular, Corollary \ref{corUpperWeakAss} shows (1) $\| \pi_n - \tilde{\pi}_n \| \rightarrow 0$ whenever $\alpha_n t_{\mix}^{(n)}(\epsilon) \rightarrow 0$ and $\epsilon < 1/2$, and (2) $\limsup \| \pi_n - \tilde{\pi}_n \| \leq 1 - e^{-c} + \epsilon$ whenever $\alpha_n t_{\mix}^{(n)}(\epsilon) \rightarrow c$. We then sharpen these bounds (i.e.\ we remove the $\epsilon < 1/2$ assumption in (1) and the additive $\epsilon$ in (2)) under the cutoff assumption (which is needed for our lower bounds) to obtain the upper bounds of the theorem (see Corollary \ref{corFinalUpper}).
\item For the lower bounds, we first prove a non-asymptotic result in terms of a certain notion of high-probability hitting times, assuming only ergodicity (see Lemma \ref{lemLowerWeaker}). We then invoke the aforementioned result from \cite{basu2015characterization}, which relates these hitting times to the \textit{relaxation time} $t_{\rel}^{(n)}$ (see \eqref{eqDefnRelaxTime}), assuming laziness and reversibility. Together, this yields bounds of the form
\begin{equation}
\| \pi_n - \tilde{\pi}_n \| \gtrapprox 1 - \exp \left( - \alpha_n \left(   t_{\mix}^{(n)}  -  t_{\rel}^{(n)}  \right) \right)
\end{equation}
assuming only laziness and reversibility (see \eqref{eqLowerWithoutCutoff}). Finally, we use the known fact $t_{\rel}^{(n)} = o ( t_{\mix}^{(n)} )$ under pre-cutoff to obtain the lower bounds of the theorem (see Corollaries \ref{corLowerWithDeltas}-\ref{corLowerWithoutDeltas}).
\end{itemize}

Our second main result concerns pre-cutoff. As alluded to above, we believe some notion of cutoff is necessary for lower bounds like those above. Indeed, in Theorem \ref{thmEquivalence} we show that for lazy and reversible chains, pre-cutoff \eqref{eqDefnPrecutoffIntro} implies a certain perturbation condition, and
\begin{equation} \label{eqPrecutFailNiceIntro}
\sup_{\epsilon \in (0,1/2)} \liminf_{n \rightarrow \infty} \frac{ t_{\mix}^{(n)}(\epsilon) }{ t_{\mix}^{(n)}(1-\epsilon) } = \infty 
\end{equation}
(which is slightly stronger than the negation of \eqref{eqDefnPrecutoffIntro}) implies the negation of the perturbation condition. Roughly speaking, this condition is as follows: for certain sequences $\{ \alpha_{n,\epsilon} \}_{n \in \N, \epsilon \in (0,1/2)} \subset (0,1)$ and all $\epsilon \in (0,1/2)$, there exists a sequence of restart perturbations with restart probabilities $\{ \alpha_{n,\epsilon} \}_{n \in \N}$ and stationary distributions $\{ \tilde{\pi}_{n,\epsilon} \}_{n \in \N}$ such that $\| \pi_n - \tilde{\pi}_{n,\epsilon} \| \rightarrow 1$. Hence, Theorem \ref{thmEquivalence} says that chains with pre-cutoff are sensitive to perturbation, in the sense that certain perturbations maximally change the stationary distribution, and the converse (almost) holds. The only gap in our logic involves the case
\begin{equation} \label{eqBizarreChainsIntro}
\sup_{\epsilon \in (0,1/2)} \liminf_{n \rightarrow \infty} \frac{ t_{\mix}^{(n)}(\epsilon) }{ t_{\mix}^{(n)}(1-\epsilon) } < \infty = \sup_{\epsilon \in (0,1/2)} \limsup_{n \rightarrow \infty} \frac{ t_{\mix}^{(n)}(\epsilon) }{ t_{\mix}^{(n)}(1-\epsilon) } .
\end{equation}
We believe this case primarily involves pathological sequences of chains -- for example, when $P_n$ corresponds to Chain 1 from Figure \ref{fig_intro} when $n$ is even and Chain 2 when $n$ is odd.

\subsection{Context of main results}

Having previewed our results, we provide a brief overview of related work; see Section \ref{secRelated} for further details.

Theorem \ref{thmAllRegimes} says that a threshold phenomena for the original chain -- cutoff -- translates into a different threshold phenomena for the perturbed chain -- the trichotomy shown above. Another point of interest is that similar trichotomies have been established in several recent papers. For example, \cite{caputo2019mixing} shows that the restart perturbation adopts the cutoff behavior of the original chain if $\alpha_n t_{\mix}^{(n)}(\epsilon) \rightarrow 0$, has a distinct convergence to stationarity if $\alpha_n t_{\mix}^{(n)}(\epsilon) \rightarrow \infty$, and exhibits an intermediate behavior if $\alpha_n t_{\mix}^{(n)}(\epsilon) \rightarrow (0,\infty)$, assuming the original chain is the simple random walk on a particular random graph. Similar results were obtained in \cite{avena2018random,avena2020linking} for random walks on dynamic random graphs; here $\alpha_n$ denotes the ``rate of change'' of edges. Finally, \cite{vial2019structural} studies the matrix $\{ \tilde{\pi}_{n,i} \}_{i \in [n]}$, where $\tilde{\pi}_{n,i}$ corresponds to restarting at state $i$; the authors prove this matrix has effective dimension $O(1)$ if $\alpha_n t_{\mix}^{(n)}(\epsilon) \rightarrow 0$, conjecture the dimension is $\Omega(n/\log n)$ if $\alpha_n t_{\mix}^{(n)}(\epsilon) \rightarrow \infty$, and prove the dimension is $O(n^{f(c)})$ for some $f(c) \in (0,1)$ if $\alpha_n t_{\mix}^{(n)}(\epsilon) \rightarrow c \in (0,\infty)$, when the original chain is generated as in \cite{caputo2019mixing}.

Ultimately, this work, \cite{caputo2019mixing}, \cite{avena2018random}, \cite{avena2020linking}, and \cite{vial2019structural} all study different questions, but the similarities speak to a much deeper phenomena: some aspect of the original chain is unaffected when $\alpha_n t_{\mix}^{(n)}(\epsilon) \rightarrow 0$, this aspect is significantly altered when $\alpha_n t_{\mix}^{(n)}(\epsilon) \rightarrow \infty$, and an intermediate behavior occurs when $\alpha_n t_{\mix}^{(n)}(\epsilon) \rightarrow (0,\infty)$. However, in contrast to \cite{caputo2019mixing}, \cite{avena2018random}, \cite{avena2020linking}, and \cite{vial2019structural}, we work directly with the stationary distribution, which is arguably the most fundamental such aspect one would hope to understand. Additionally, unlike these works, we do not assume a random graph model for the original chain; in this sense, our results are more general, while demonstrating a similar idea.

The main utility of Theorem \ref{thmEquivalence} is that, while different notions of cutoff have been proven for many different chains (such as Markovian models of card shuffling and random walks on certain graphs), there is little general theory. In fact, only recently was an abstract condition equivalent to cutoff determined in \cite{basu2015characterization} (this being a certain notion of ``hitting time cutoff''). Additionally, while Theorem \ref{thmEquivalence} relies on an inequality from \cite{basu2015characterization}, it is much more than a corollary of this result, and both our notion of cutoff (pre-cutoff instead of cutoff) and our equivalent notion (perturbation sensitivity instead of hitting times) differ. Finally, we note that in addition to \cite{basu2015characterization}, the very recent paper \cite{salez2021cutoff} proves cutoff for all chains with a certain notion of nonnegative curvature using an ``entropic concentration principle.'' The authors speculate that this principle underlies cutoff in general, but this has not been formally established.

In short, this paper contributes to two lines of work. First, we add to the collection of ``trichotomy'' results, but we study the stationary distribution directly and do not assume a random graph model. Second, we add to the general theory of cutoff in a similar vein to \cite{basu2015characterization}, but for a different notion of cutoff and a different equivalent notion.

\subsection{Organization}

The remainder of the paper is organized as follows. We begin in Section \ref{secPrelim} with definitions. Section \ref{secResults} contains the two theorems described above. We discuss the example chains from Figure \ref{fig_intro} in Section \ref{secExamples}. Section \ref{secRelated} discusses related work in more detail. Sections \ref{secProofTrichotomy}, \ref{secProofEquivalence}, and \ref{secProofExampleMain} contain proofs, with some calculations deferred to the appendix.

\section{Preliminaries} \label{secPrelim}

We begin with some notation. Let $\Z_+$ denote the set of nonnegative integers, and let $\{ X_n(t) \}_{t \in \Z_+}$ be a time-homogeneous, irreducible, and aperiodic Markov chain with state space $[n] = \{1,\ldots,n\}$. We denote by $P_n$ the transition matrix of this chain, i.e.\ the matrix with $(i,j)$-th entry
\begin{equation}
P_n(i,j) = \P ( X_n(t+1) = j | X_n(t) = i )\ \forall\ i, j \in [n], t \in \Z_+ .
\end{equation}
It is a standard result that this chain has a unique stationary distribution $\pi_n$, i.e.\ a unique vector $\pi_n$ satisfying $\pi_n = \pi_n P_n$ and $\sum_{i=1}^n \pi_n(i) = 1$. Here and moving forward, we treat all vectors as row vectors. For $i \in [n]$, we let $e_i$ denote the length-$n$ vector with 1 in the $i$-th coordinate and zeros elsewhere. Also, we let $\Delta_{n-1}$ denote the set of distributions on $[n]$, so that (for example) $\pi_n \in \Delta_{n-1}$. Finally, we let $\mathcal{E}_n$ denote the set of transition matrices for time-homogeneous, irreducible, and aperiodic Markov chains with state space $[n]$, so that (for example) $P_n \in \mathcal{E}_n$.

Some of our results will only apply to a strict subset of $\mathcal{E}_n$. In particular, at times we require the chain to be lazy, meaning $P_n(i,i) \geq 1/2\ \forall\ i \in [n]$, and reversible, meaning $\pi_n(i) P_n(i,j) = \pi_n(j) P_n(j,i)\ \forall\ i,j \in [n]$. We note that any chain can be made lazy without changing its stationary distribution; namely, by considering $(P_n+I_n)/2$ instead of $P_n$, where $I_n$ is the $n \times n$ identity matrix. In this sense, reversibility is the most restrictive of our assumptions. However, this is a fairly common restriction in the mixing times literature, since it guarantees the eigenvalues of $P_n$ are real and allows one to use certain linear algebraic techniques (see e.g.\ Chapter 12 of \cite{levin2009markov}).

We next define the distance between the $t$-step distribution and stationarity by
\begin{equation} \label{eqDefnDnT}
d_n(t) = \max_{i \in [n]} \| e_i P_n^t - \pi_n \|\ \forall\ t \in \Z_+ ,
\end{equation}
where $\| \cdot \|$ denotes total variation distance, $\| \mu - \nu \| = \max_{A \subset [n]} | \mu(A) - \nu(A) |$ for $\mu, \nu \in \Delta_{n-1}$. For $\epsilon \in (0,1)$, we then define the $\epsilon$-mixing time as
\begin{equation} \label{eqDefnTmix}
t_{\mix}^{(n)}(\epsilon) = \min \{ t \in \Z_+ : d_n(t) \leq \epsilon \} .
\end{equation}
As is convention in the literature, we set $t_{\mix}^{(n)} = t_{\mix}^{(n)}(1/4)$. We also note the following monotocity property, which follows immediately from the definition:
\begin{equation} \label{eqMixMonotone}
\forall\ \epsilon, \delta \in (0,1) \textrm{ s.t.\ } \epsilon \leq \delta,\ t_{\mix}^{(n)}(\epsilon) \geq t_{\mix}^{(n)} (\delta) .
\end{equation}
We next recall the two notions of cutoff from Section \ref{secIntro}. First, a sequence $\{ P_n \}_{n \in \N}$ with $P_n \in \mathcal{E}_n\ \forall\ n \in \N$ is said to exhibit cutoff if
\begin{equation} \label{eqDefnCutoff}
\lim_{n \rightarrow \infty}  \frac{t_{\mix}^{(n)}(\epsilon)}{ t_{\mix}^{(n)}(1-\epsilon) } = 1\ \forall\ \epsilon \in (0,1/2) .
\end{equation} 
A basic result (see e.g.\ Section 18.1 of \cite{levin2009markov}) says that cutoff occurs if and only if
\begin{equation} \label{eqStepFunction}
\lim_{n \rightarrow \infty} d_n ( s t_{\mix}^{(n)} )  = \begin{cases} 1 , & s < 1 \\ 0 , & s > 1 \end{cases} .
\end{equation}
Thus, cutoff means the graph of $d_n(t)$ approaches a step function as $n \rightarrow \infty$, when the $t$-axis is normalized by $t_{\mix}^{(n)}$. Put differently, the chain is quite far from stationarity at time e.g.\ $0.99 t_{\mix}^{(n)}$, then suddenly reaches stationarity at time e.g.\ $1.01 t_{\mix}^{(n)}$. The weaker notion of pre-cutoff only requires
\begin{equation} \label{eqDefnPreCutoff}
\sup_{\epsilon \in (0,1/2)} \limsup_{n \rightarrow \infty} \frac{ t_{\mix}^{(n)}(\epsilon) }{ t_{\mix}^{(n)}(1-\epsilon)}  < \infty .
\end{equation}

For the perturbation analysis described in the introduction, it will be convenient to introduce some additional notation. First, given $P_n \in \mathcal{E}_n$ and $\alpha \in (0,1)$, we define 
\begin{equation} \label{eqDefnPnBall}
B(P_n,\alpha) = \left\{ \tilde{P}_n \in \mathcal{E}_n : \max_{i \in [n]} \| e_i P_n - e_i \tilde{P}_n \| \leq \alpha  \right\} .
\end{equation}
In words, $B(P_n,\alpha)$ is the set of transition matrices for time-homogeneous, irreducible, and aperiodic chains whose rows differ from the rows of $P_n$ by at most $\alpha$ (in the total variation distance). We will denote the unique stationary distribution of $\tilde{P}_n \in B(P_n,\alpha)$ by $\tilde{\pi}_n$. A particular subset of $B(P_n,\alpha)$ is the class of restart perturbations discussed in the introduction. Such perturbations have the form
\begin{equation}
(1-\alpha) P_n + \alpha 1_n^{\trans} \sigma_n \in  B(P_n,\alpha) 
\end{equation}
for some $\alpha \in (0,1)$ and $\sigma_n \in \Delta_{n-1}$, where $1_n$ is the length-$n$ row vector of ones. The corresponding chain $\{ \tilde{X}_n(t) \}_{t \in \Z_+}$ has the following dynamics: given $\tilde{X}_n(t)$, flip a coin that lands heads with probability $\alpha$; if heads, let $\tilde{X}_n(t+1) \sim \sigma_n$ (i.e.\ restart at a state sampled from $\sigma_n$); if tails, let $\tilde{X}_n(t+1) \sim e_{\tilde{X}_n(t)} P_n$ (i.e.\ follow the original chain). Note that such perturbations only depend on the restart probability $\alpha$ and the restart distribution $\sigma_n$. Thus, we use
\begin{equation}
P_{\alpha,\sigma_n} = (1-\alpha) P_n + \alpha 1_n^{\trans} \sigma_n 
\end{equation}
to define restart perturbations. We denote the corresponding stationary distribution by $\pi_{\alpha,\sigma_n}$. Moving forward, $\alpha$ will typically depend on $n$, in which case we write $P_{\alpha_n,\sigma_n}$ and $\pi_{\alpha_n,\sigma_n}$.

Finally, we note the following (standard) notation for $\{ a_n \}_{n \in \N} , \{ b_n \}_{n \in \N} \subset [0,\infty)$ will be used: we write $a_n = O(b_n)$, $a_n = \Omega(b_n)$, $a_n = \Theta(b_n)$, and $a_n = o(b_n)$, respectively, if $\limsup_{n \rightarrow \infty} a_n / b_n < \infty$, $\liminf_{n \rightarrow \infty} a_n / b_n > 0$, $a_n = O(b_n)$ and $a_n = \Omega(b_n)$, and $\lim_{n \rightarrow \infty} a_n / b_n = 0$, respectively.

\section{Main results} \label{secResults}

\subsection{Trichotomy}

We can now formally state the trichotomy described in Section \ref{secIntro}. 

\begin{theorem}[Trichotomy] \label{thmAllRegimes}
Let $P_n \in \mathcal{E}_n, \alpha_n \in (0,1)\ \forall\ n \in \N$, and let $\epsilon \in (0,1)$ be independent of $n$. Assume $\{ P_n \}_{n \in \N}$ exhibits cutoff, each $P_n$ is lazy and reversible, and $\lim_{n \rightarrow \infty} \alpha_n t_{\mix}^{(n)}(\epsilon) = c \in [0,\infty]$. Then the following hold:
\begin{itemize} 
\item If $c = 0$, then $\forall\ \{ \tilde{P}_n \}_{n \in \N}$ s.t.\ $\tilde{P}_n \in B(P_n,\alpha_n)\ \forall\ n \in \N$,
\begin{equation}
\lim_{n \rightarrow \infty} \| \pi_n - \tilde{\pi}_n \| = 0  .
\end{equation}
\item If $c \in (0,\infty)$, then $\forall\ \{ \tilde{P}_n \}_{n \in \N}$ s.t.\ $\tilde{P}_n \in B(P_n,\alpha_n)\ \forall\ n \in \N$,
\begin{equation} 
\limsup_{n \rightarrow \infty} \| \pi_n - \tilde{\pi}_n \| \leq 1 - e^{-c}  .
\end{equation}
Furthermore, the bound is tight, i.e.\ $\exists\ \{ \tilde{P}_n \}_{n \in \N}$ s.t.\ $\tilde{P}_n \in B(P_n,\alpha_n)\ \forall\ n \in \N$ and
\begin{equation} 
\liminf_{n \rightarrow \infty} \| \pi_n - \tilde{\pi}_n \| \geq  1 - e^{-c}  .
\end{equation}
In particular, $\tilde{P}_n$ is a restart perturbation, i.e.\ $\tilde{P}_n = P_{\alpha_n,\sigma_n}$ for some $\sigma_n \in \Delta_{n-1}$.
\item If $c = \infty$, then $\exists\ \{ \tilde{P}_n \}_{n \in \N}$ s.t.\ $\tilde{P}_n \in B(P_n,\alpha_n)\ \forall\ n \in \N$ and
\begin{equation} \label{eqCorInfRegime}
\lim_{n \rightarrow \infty} \| \pi_n - \tilde{\pi}_n \| =  1  .
\end{equation}
In particular, $\tilde{P}_n$ is a restart perturbation, i.e.\ $\tilde{P}_n = P_{\alpha_n,\sigma_n}$ for some $\sigma_n \in \Delta_{n-1}$.
\end{itemize}
\end{theorem}

Note that, given any $c' \in [0,1]$, the theorem guarantees the existence of a sequence of (restart) perturbations $\{ \tilde{P}_n \}_{n \in \N}$ s.t.\ $\lim_{n \rightarrow \infty} \| \pi_n - \tilde{\pi}_n \|  = c'$. This holds without explicit knowledge of $P_n$, and thus we may lack expressions for (or even estimates of) $\pi_n$ and $\tilde{\pi}_n$; nevertheless, we obtain a precise asymptotic comparison of these distributions.

The proof of Theorem \ref{thmAllRegimes} can be found in Section \ref{secProofTrichotomy}, which contains several intermediate results that require weaker assumptions than the theorem. We mention the key ideas here:
\begin{itemize}
\item For the upper bounds (see Section \ref{secProofUpper}), we first use the well-known relationship between total variation and coupling to obtain non-asymptotic bounds for $\| \pi_n - \tilde{\pi}_n \|$ in terms of $\alpha_n$ and the distance from stationarity function \eqref{eqDefnDnT} (see Lemma \ref{lemGenericUpper}). These bounds follow by constructing a (possibly suboptimal) coupling of the $t$-step distributions $(\tilde{\pi}_n P_n^t , \tilde{\pi}_n \tilde{P}_n^t)$ from a sequence of $t$ optimal couplings (one per step). From these bounds, we derive two asymptotic bounds under weak assumptions: $\| \pi_n - \tilde{\pi}_n \| \rightarrow 0$ assuming ergodicity and $\epsilon < 1/2$ when $c = 0$, and $\| \pi_n - \tilde{\pi}_n \| \rightarrow 1 - e^{-c} + \epsilon$ assuming ergodicity and $t_{\mix}^{(n)}(\epsilon) \rightarrow \infty$ when $c \in (0,\infty)$ (see Corollary \ref{corUpperWeakAss}). Finally, we use the cutoff assumption (which is needed for the lower bounds) to sharpen these results and obtain the bounds of the theorem.
\item For the lower bounds (see Section \ref{secProofLower}), we show that a certain notion of high-probability hitting times, denoted $t_{\hit}^{(n)}$ and defined in \eqref{eqDefnThit}, can be used to construct a restart perturbation for which $\| \pi_n - \tilde{\pi}_n \|$ is lower bounded in terms of the product $\alpha_n t_{\hit}^{(n)}$. This is possible owing to the relatively simple structure of restart perturbations, for which $\tilde{\pi}_n$ can be expressed only in terms of $\alpha_n$ and $P_n$ (see \eqref{eqInfWeakSum}). Moreover, the aforementioned \cite{basu2015characterization} shows that $t_{\hit}^{(n)}$ can be bounded in terms of $t_{\mix}^{(n)}$, and these bounds are nontrivial assuming at least pre-cutoff holds. Taken together, we can obtain nontrivial lower bounds for $\| \pi_n - \tilde{\pi}_n \|$ in terms of $\alpha_n t_{\mix}^{(n)}$ (as in the theorem), assuming $c = \infty$ and pre-cutoff holds, or $c \in (0,\infty)$ and cutoff holds (see Corollary \ref{corLowerWithDeltas}). Finally, similar to the upper bound analysis, we sharpen the former bound under the cutoff assumption to obtain the results in the theorem.
\end{itemize}

\subsection{Pre-cutoff ``equivalence"} \label{secResultsEquiv}

We next turn to Theorem \ref{thmEquivalence}. As discussed in the introduction, the theorem provides a near-equivalence between pre-cutoff and a certain perturbation condition. More specifically, we will show that pre-cutoff implies a certain perturbation condition, and that this condition fails whenever
\begin{equation} \label{eqPrecutFailNice}
\sup_{\epsilon \in (0,1/2)} \liminf_{n \rightarrow \infty} \frac{t_{\mix}^{(n)}(\epsilon) }{ t_{\mix}^{(n)}(1-\epsilon) } = \infty .
\end{equation}
The caveat of Theorem \ref{thmEquivalence} being a near-equivalence arises because \eqref{eqPrecutFailNice} is stronger than the negation of pre-cutoff; one can construct sequences of chains for which pre-cutoff and \eqref{eqPrecutFailNice} both fail. For instance, in Section \ref{secExamples} we provide two examples with drastically different cutoff behaviors; if we construct a new sequence that oscillates between these two, we obtain
\begin{equation}
\liminf_{n \rightarrow \infty} \frac{t_{\mix}^{(n)}(\epsilon) }{ t_{\mix}^{(n)}(1-\epsilon) }= 1 , \quad \limsup_{n \rightarrow \infty} \frac{t_{\mix}^{(n)}(\epsilon) }{ t_{\mix}^{(n)}(1-\epsilon) } = \infty , \quad \forall\ \epsilon \in (0,1/2) .
\end{equation}
However, this oscillating sequence is pathological; the literature typically considers sequences $\{P_n\}_{n \in \N}$ for which $P_n$ is defined in the same manner for each $n$ (i.e., $P_n$ defines a card shuffling model for $n$ cards for each $n$, a random walk on a graph of $n$ nodes for each $n$, etc.). Thus, the ``near-equivalence'' caveat is a small one.

Before presenting Theorem \ref{thmEquivalence}, we must define the perturbation condition. However, this condition is somewhat opaque, so for expository purposes, we first explain the difficulty in deriving it. We begin with the most obvious candidate, the condition from Theorem \ref{thmAllRegimes}:
\begin{equation} \label{eqEqCondFirstAttempt}
\forall\ \{ \alpha_n \}_{n \in \N} \textrm{ s.t.\ } \lim_{n \rightarrow \infty} \alpha_n t_{\mix}^{(n)}(\epsilon) = \infty,\ \exists\ \{ \tilde{P}_n \}_{n \in \N} \textrm{ s.t.\ } \lim_{n \rightarrow \infty} \| \pi_n - \tilde{\pi}_n \| = 1 . 
\end{equation}
Indeed, in the proof of Theorem \ref{thmAllRegimes} (see Section \ref{secProofUpper}), we show that pre-cutoff implies \eqref{eqEqCondFirstAttempt} (assuming laziness and reversibility). The difficulty arises in showing that \eqref{eqEqCondFirstAttempt} fails whenever \eqref{eqPrecutFailNice} holds. The most obvious approach is as follows. When \eqref{eqPrecutFailNice} holds, it is \textit{possible} that for some fixed $\epsilon \in (0,1/2)$, 
\begin{equation} \label{eqEqCondFirstAttemptPossible}
\lim_{n \rightarrow \infty}  \frac{t_{\mix}^{(n)}(\epsilon) }{ t_{\mix}^{(n)}(1-\epsilon) } = \infty ,
\end{equation}
which suggests setting $\alpha_n = c / t_{\mix}^{(n)}(1-\epsilon)$ for some $c$ independent of $n$, since then
\begin{equation}\label{eqEqCondFirstAttemptBeforeInvoke}
 \lim_{n \rightarrow \infty} \alpha_n t_{\mix}^{(n)}(\epsilon) =c \lim_{n \rightarrow \infty} \frac{t_{\mix}^{(n)}(\epsilon) }{ t_{\mix}^{(n)}(1-\epsilon) }  = \infty .
\end{equation}
Our task would then be reduced to upper bounding $\| \pi_n - \tilde{\pi}_n \|$ (perhaps via techniques used for upper bounds above). Unfortunately, \eqref{eqEqCondFirstAttemptPossible} is not guaranteed to hold.

While this first attempt fails, it illustrates the dissonance at hand: \eqref{eqEqCondFirstAttempt} considers sequences $\{ \alpha_n \}_{n \in \N}$ depending only on $n$; the analogous sequence $\{ 1 / t_{\mix}^{(n)}(1-\epsilon) \}_{n \in \N, \epsilon \in (0,1/2)}$ in \eqref{eqPrecutFailNice} depends on both $n$ and $\epsilon$. Hence, as a second attempt, we could modify \eqref{eqEqCondFirstAttempt} to involve a sequence $\{ \alpha_{n,\epsilon} \}_{n \in \N, \epsilon \in (0,1/2)}$  depending on both $n$ and $\epsilon$. However, if \eqref{eqEqCondFirstAttempt} is modified in this manner, it is no longer implied by pre-cutoff via the proof of Theorem \ref{thmAllRegimes}.

It turns out this issue can be resolved by placing appropriate restrictions on the set of sequences of restart probabilities appearing in the perturbation condition. In particular, we will say that the sequence of restart probabilities $\{ \alpha_{n,\epsilon} \}_{n \in \N, \epsilon \in (0,1/2)} \subset (0,1)$ \textit{coincides with} the mixing times $\{ t_{\mix}^{(n)}(\epsilon) \}_{n \in \N, \epsilon \in (0,1)}$ if
\begin{gather} \label{eqDefnCoincides}
\sup_{\epsilon \in (0,1/2)} \liminf_{n \rightarrow \infty} \alpha_{n,\epsilon} t_{\mix}^{(n)} (\epsilon) = \infty , \\ \frac{ \alpha_{n,\epsilon} }{ \alpha_{n,\delta} } \in \left[ \frac{ t_{\mix}^{(n)}(1-\delta) }{ t_{\mix}^{(n)}(1-\epsilon) } , 1 \right]\ \forall\ \epsilon, \delta \in (0,1/2) \textrm{ s.t.\ } \epsilon \geq \delta, \forall\ n \in \N ,
\end{gather}
and we will restrict to sequences that coincide with the mixing times (such sequences exist under the assumption of laziness; see Section \ref{secProofEquivalence}) More specifically, we define the following perturbation condition for use in our second main result.
\begin{condition} \label{condSens}
For any sequence of restart probabilities $\{ \alpha_{n,\epsilon} \}_{n \in \N, \epsilon \in (0,1/2)} \subset (0,1)$ that coincides with the mixing times $\{ t_{\mix}^{(n)}(\epsilon) \}_{n \in \N, \epsilon \in (0,1)}$, there exists $\{ \sigma_{n,\epsilon} \}_{n \in \N,\epsilon \in (0,1/2)}$ such that
\begin{equation}
\sigma_{n,\epsilon} \in \Delta_{n-1}\ \forall\ n \in \N, \epsilon \in (0,1/2), \quad  \lim_{n \rightarrow \infty} \| \pi_n - \pi_{\alpha_{n,\epsilon},\sigma_{n,\epsilon}} \| = 1\ \forall\ \epsilon \in (0,1/2) .
\end{equation}
\end{condition}
We note the definition of ``coincides with'' yields the following property: when pre-cutoff holds and $\{ \alpha_{n,\epsilon} \}_{n \in \N, \epsilon \in (0,1/2)}$ coincides with the mixing times, $\alpha_{n,\epsilon} t_{\mix}^{(n)} (\epsilon) \rightarrow \infty\ \forall\ \epsilon \in (0,1/2)$. In words, not only is the supremum in \eqref{eqDefnCoincides} infinite, the $\liminf$ in \eqref{eqDefnCoincides} is infinite, for every $\epsilon \in (0,1/2)$. This allows us to prove (using ideas from Theorem \ref{thmAllRegimes}) that Condition \ref{condSens} is implied by pre-cutoff, while also proving that Condition \ref{condSens} fails (via the approach discussed above) whenever \eqref{eqPrecutFailNice} holds. With Condition \ref{condSens} in place, we present Theorem \ref{thmEquivalence}.
\begin{theorem}[Pre-cutoff ``equivalence"] \label{thmEquivalence}
Let $\{ P_n \}_{n \in \N}$ be such that $P_n \in \mathcal{E}_n$ is lazy and reversible for each $n \in \N$. If $\{ P_n \}_{n \in \N}$ exhibits pre-cutoff, then Condition \ref{condSens} holds; if $\{ P_n \}_{n \in \N}$ satisfies \eqref{eqPrecutFailNice}, then Condition \ref{condSens} fails.
\end{theorem}
\begin{proof}
See Section \ref{secProofEquivalence}.
\end{proof}

\section{Illustrative examples} \label{secExamples}

Our results suggest a deep connection between some notion of cutoff and some notion of perturbation sensitivity. Here we illustrate this with two example chains called the \textit{winning streak reversal} (WSR) and the \textit{complete graph bijection} (CGB) (denoted Chain 1 and Chain 2, respectively, in Figure \ref{fig_intro} from the introduction).

\subsection{Definitions}

\begin{figure}
\centering
\includegraphics[height=1in]{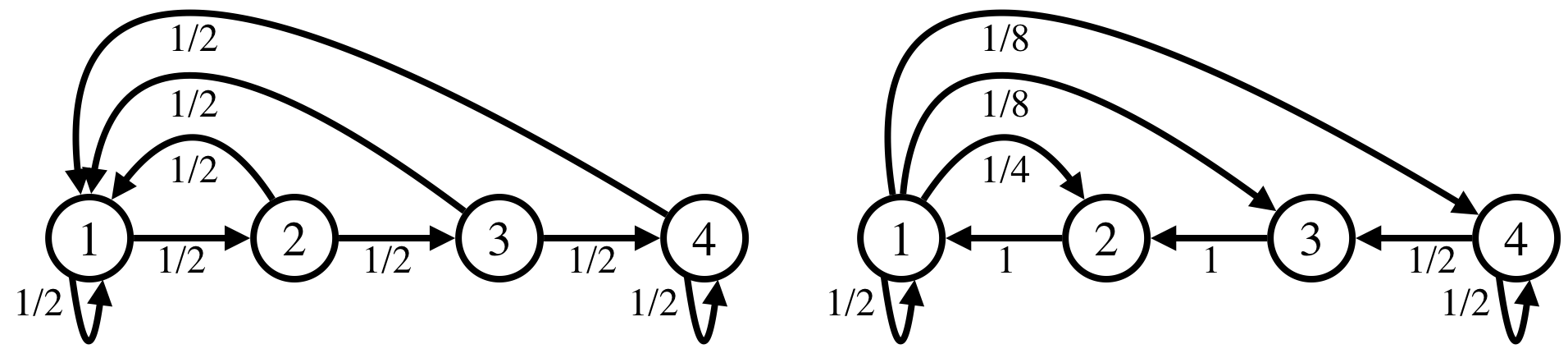}
\caption{Winning streak chain (left) and its reversal (right) for $n = 4$.} \label{fig_ex_wsr}
\end{figure}
\begin{figure}
\centering
\includegraphics[height=1in]{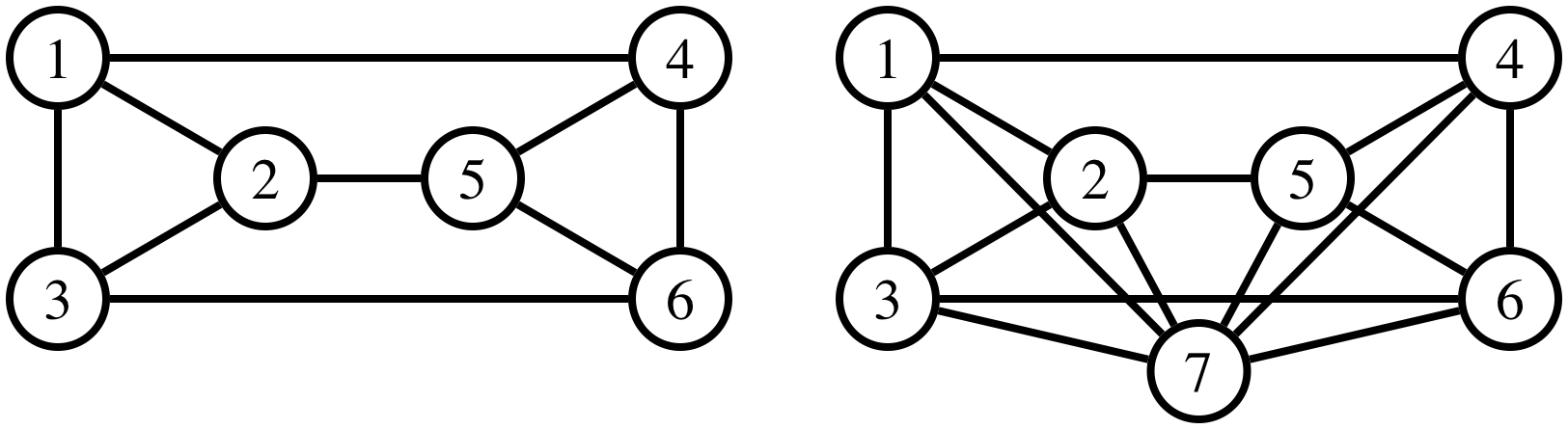}
\caption{Complete graph bijection for $n=6$ (left) and $n=7$ (right).} \label{fig_ex_cgb}
\end{figure}

The WSR is taken from \cite{levin2009markov}. As its name suggests, this chain is the time reversal of the so-called \textit{winning streak} chain. The winning streak chain is shown at left in Figure \ref{fig_ex_wsr} and has the following interpretation. At each step, one plays a fair game. If the game is won, the winning streak is increased, meaning the state is increased by 1 (unless the current state is $n$, in which case the state remains $n$); if the game is lost, the winning streak ends, meaning the state returns to its lowest value. (Given this interpretation, it is more sensible to use state space $\{0,\ldots,n-1\}$, but for consistency with the rest of this paper, we use state space $\{1,\ldots,n\}$.) The reversal of this chain, which we analyze, has transition matrix and stationary distribution
\begin{gather} \label{eqWsrTransAndStat}
P_n(i,j) = \begin{cases} 2^{-j} , & i = 1, j \in \{1,\ldots,n-1\} \\ 2^{-n+1} , & i=1, j = n \\ 1 , & i \in \{2,\ldots,n-1\}, j = i-1 \\ 2^{-1} , & i = n, j \in \{n-1,n\}  \end{cases} , 
\pi_n(i) = \begin{cases} 2^{-i} , & i \in \{1,\ldots,n-1\} \\  2^{-n+1} , & i = n \end{cases} \end{gather}
Note that $P_n(1,i) = \pi_n(i)\ \forall\ i \in [n]$; hence, the chain started from state 1 reaches stationarity (exactly) after 1 step. Furthermore, the chain starting from $i \in \{2,\ldots,n-1\}$ deterministically transitions to state 1 in $i-1$ steps and thus reaches stationarity (again, exactly) after $i$ steps. As will be seen, this implies a particularly strong form of cutoff.

For the CGB, we construct complete graphs on nodes $\{1,\ldots,n/2\}$ and $\{1+n/2,\ldots,n\}$; we then add edges between $i$ and $i+n/2$ for each $i \in [n/2]$, corresponding to the bijection $i \mapsto i+n/2$. For $n$ odd, we construct this graph for $n-1$, then add an auxiliary node $n$, along with an edge between $n$ and every $i \in [n-1]$. Figure \ref{fig_ex_cgb} shows these graphs for $n=6$ and $n=7$. We consider the lazy random walks on these graphs. The transition matrices are
\begin{gather} \label{eqCgbTrans}
P_n = \frac{I_n}{2} + \frac{1}{n} \begin{bmatrix} 1_{\frac{n}{2}}^{\trans} 1_{\frac{n}{2}} - I_{\frac{n}{2}} & I_{\frac{n}{2}} \\ I_{\frac{n}{2}} & 1_{\frac{n}{2}}^{\trans} 1_{\frac{n}{2}} - I_{\frac{n}{2}} \end{bmatrix}\ \forall\ n \textrm{ even}, \\
P_n = \frac{1}{2} I_n + \frac{1}{n+1} \begin{bmatrix} 1_{\frac{n-1}{2}}^{\trans} 1_{\frac{n-1}{2}} - I_{\frac{n-1}{2}} & I_{\frac{n-1}{2}} & 1_{\frac{n-1}{2}}^{\trans} \\ I_{\frac{n-1}{2}} & 1_{\frac{n-1}{2}}^{\trans} 1_{\frac{n-1}{2}} - I_{\frac{n-1}{2}} & 1_{\frac{n-1}{2}}^{\trans} \\ \frac{n+1}{2(n-1)}  1_{\frac{n-1}{2}} &\frac{n+1}{2(n-1)}   1_{\frac{n-1}{2}} & 0 \end{bmatrix}\ \forall\ n \textrm{ odd} .
\end{gather}
It is a standard result that the degree distribution is stationary for lazy random walks on undirected graphs, so we have
\begin{gather} \label{eqCgbStat}
\pi_n(i) = \frac{1}{n}\ \forall\ i \in [n], n \textrm{ even} , \quad \pi_n(i) = \begin{cases}  \frac{n+1}{(n+3)(n-1)} , & i \in [n-1] \\ \frac{2}{n+3} , & i = n \end{cases}\ \forall\ n \textrm{ odd} .
\end{gather}

\subsection{Mixing and perturbation analysis} \label{secExamplesAnalysis}

We next state a proposition that estimates the mixing times of these chains. The proposition contains several results. First, we show both chains have $\Theta(n)$ $\epsilon$-mixing time, for any fixed $\epsilon \in (0,1/2)$. Furthermore, the proposition says that for the WSR and for any such $\epsilon$,
\begin{equation} \label{eqWsrCutoffBehavior}
1 \leq \frac{t_{\mix}^{(n)}(\epsilon) }{ t_{\mix}^{(n)}(1-\epsilon) } = 1 + \Theta \left( \frac{1}{n} \right) .
\end{equation}
Hence, the ratios in \eqref{eqWsrCutoffBehavior} converge to 1 at rate $n^{-1}$, a particularly strong notion of cutoff (the basic cutoff property, \eqref{eqDefnCutoff}, imposes no rate of convergence). In contrast, for the CGB, the proposition shows that these ratios are $\Theta(n)$, the maximum (up to constants) among all chains with $\Theta(n)$ $\epsilon$-mixing times. In summary, while both chains have equivalent $\epsilon$-mixing times, their cutoff behaviors are at opposite extremes among such chains.
\begin{proposition} \label{propExamplesMixing}
Let $\epsilon \in (0,1/2)$ be independent of $n$. Then the following hold:
\begin{itemize}
\item Suppose $\{ P_n \}_{n \in \N}$ is the WSR. Then
\begin{equation}
(n - 1) - \log_2(1/\epsilon) < t_{\mix}^{(n)}(1-\epsilon) \leq n-1 , \quad t_{\mix}^{(n)}(\epsilon) = n-1 \quad \forall\ n \in \N .
\end{equation}
\item Suppose $\{ P_n \}_{n \in \N}$ is the CGB.  Then
\begin{equation}
t_{\mix}^{(n)}(1-\epsilon) = 1\ \forall\ n \in \N \textrm{ sufficiently large}, \quad t_{\mix}^{(n)}(\epsilon) = \Theta(n) .
\end{equation}
\end{itemize}
\end{proposition}
\begin{proof}
For the WSR, much of the analysis is borrowed from \cite{levin2009markov}, which we include in Appendix \ref{appProofExamplesMixing} for completeness. For the CGB, the proofs of $t_{\mix}^{(n)}(1-\epsilon) = 1$ and $t_{\mix}^{(n)}(\epsilon) = \Omega(n)$ use fairly elementary tools and are deferred to Appendix \ref{appProofExamplesMixing}; showing  $t_{\mix}^{(n)}(\epsilon) = O(n)$ requires a nontrivial coupling construction, which is done in Section \ref{secProofExampleMain}.
\end{proof}

The next proposition shows that these polarized cutoff behaviors translate into polarized perturbation behaviors. First, for the WSR, note we cannot invoke lower bounds from our earlier analysis, since we lack laziness. However, we can prove a stronger result: namely, we can identify an uncountable class of restart perturbations such that $\| \pi_n - \pi_{\alpha_n,\sigma_n} \| \rightarrow 1$ (this is a stronger result than previous lower bounds, which only guaranteed one such perturbation). On the other hand, for the CGB, we show Theorem \ref{thmAllRegimes} fails, despite all assumptions except cutoff holding.
\begin{proposition} \label{propExamplesPerturb}
Let $\epsilon \in (0,1/2)$ be independent of $n$. Then the following hold:
\begin{itemize}
\item Suppose $\{ P_n \}_{n \in \N}$ is the WSR and $\{ \alpha_n \}_{n \in \N} \subset (0,1)$ satisfies $\alpha_n =  \Theta ( n^{-c_1} )$ for some $c_1 \in (0,1)$ independent of $n$; note $\alpha_n t_{\mix}^{(n)}(\epsilon) \rightarrow \infty$ by Proposition \ref{propExamplesMixing}. Furthermore, let $\{ \sigma_n \}_{n \in \N}$ satisfy $\sigma_n \in \Delta_{n-1}\ \forall\ n \in \N$, and, for some $c_2 > 1, c_3 > 0$ independent of $n$,
\begin{equation} \label{eqPropExamplesPerturbClass}
\lim_{n \rightarrow \infty} \sum_{i=1}^{ \floor{ c_3 \alpha_n^{-c_2} } } \sigma_n(i) = 0 .
\end{equation}
Then $\lim_{n \rightarrow \infty} \| \pi_n - \pi_{\alpha_n,\sigma_n} \| = 1$.
\item Suppose $\{ P_n \}_{n \in \N}$ is the CGB and $\{ \alpha_n \}_{n \in \N} \subset (0,1)$ satisfies
\begin{equation}
\lim_{n \rightarrow \infty}  \alpha_n n = \infty , \quad \limsup_{n \rightarrow \infty} \alpha_n = \bar{\alpha} < \frac{1}{2} ;
\end{equation}
note $\alpha_n t_{\mix}^{(n)}(\epsilon) \rightarrow \infty$ by Proposition \ref{propExamplesMixing}. Then $\forall\ \{ \tilde{P}_n \}_{n \in \N}$ s.t.\ $\tilde{P}_n \in  B(P_n,\alpha_n)\ \forall\ n \in \N$, $\limsup_{n \rightarrow \infty} \| \pi_n - \tilde{\pi}_n \| \leq  \bar{\alpha} + 1/2 < 1$.
\end{itemize}
\end{proposition}
\begin{proof}
The proof is fairly straightforward so is deferred to Appendix \ref{appProofExamplesPerturb}.
\end{proof}

We have stated Proposition \ref{propExamplesPerturb} in some generality, so it is useful to consider an example. Namely, let $\epsilon \in (0,1/2)$ and $\alpha_n = 1 / \sqrt{n}\ \forall\ n \in \N$, so that $\alpha_n t_{\mix}^{(n)}(\epsilon) \rightarrow \infty$ for both example chains. Then for the WSR, many sequences $\{ \sigma_n \}_{n \in \N}$ yield restart perturbations satisfying $\| \pi_n - \pi_{\alpha_n,\sigma_n} \| \rightarrow 1$. Some examples (easily verified to satisfy \eqref{eqPropExamplesPerturbClass}) are as follows:
\begin{itemize}
\item Uniform restart, i.e.\ $\sigma_n(i) = 1/n\ \forall\ i \in [n]$.
\item ``Flipped'' stationary restart, i.e.\ $\sigma_n(i) = \pi_n(n-i+1)\ \forall\ i \in [n]$. 
\item Deterministic restart on $\Omega( n^{3/4} )$, i.e.\ $\sigma_n = e_{i_n}$ for some $i_n = \Omega ( n^{3/4} )$.
\end{itemize}
In contrast, for this choice of $\alpha_n$ and any perturbation of the CGB, Proposition \ref{propExamplesPerturb} implies that $\limsup_{n \rightarrow \infty} \| \pi_n - \tilde{\pi}_n \| \leq 1/2$. Thus, while many restart perturbations maximally perturb the WSR, no perturbation (restart or otherwise) can maximally perturb the CGB.

\section{Related work} \label{secRelated}

We now return to discuss the existing trichotomy results mentioned in the introduction, those from \cite{caputo2019mixing,avena2018random,avena2020linking,vial2019structural}. All of these works consider the \textit{directed configuration model} (DCM), a means of constructing a graph from a given degree sequence via random edge pairings, and variants thereof It was recently shown that for random walks on sparse DCMs, and for a wider class of sparse randomly-generated chains, cutoff occurs at $\Theta(\log n)$ steps \cite{bordenave2016cutoff,bordenave2018random}. More precisely, \cite{bordenave2016cutoff,bordenave2018random} prove a probablistic analogue of \eqref{eqStepFunction}, namely
\begin{equation} \label{eqStepFunctionRandom}
d_n ( s t_{\ent}^{(n)} ) \xrightarrow[n \rightarrow \infty]{\P} \begin{cases} 1 , & s < 1 \\ 0 ,& s > 1 \end{cases} ,
\end{equation}
where $t_{\ent}^{(n)} = \Theta(\log n)$ is defined in terms of the given degrees (see Section 1 of \cite{bordenave2018random} and equation (1) of \cite{bordenave2016cutoff} for details) and $\xrightarrow[]{\P}$ denotes convergence in probability. 

Using these results, Theorem 2 in \cite{caputo2019mixing} states that for certain sequences of distributions $\{ \sigma_n \}_{n \in \N}$, the distance to stationarity $d_{\alpha_n,\sigma_n}(\cdot)$ corresponding to $P_{\alpha_n,\sigma_n}$ satisfies the following:
\begin{itemize}
\item If $\alpha_n t_{\ent}^{(n)} \rightarrow 0$,  \eqref{eqStepFunctionRandom} holds with $d_n(\cdot)$ replaced by $d_{\alpha_n,\sigma_n}(\cdot)$, i.e.\ $d_{\alpha_n,\sigma_n}(\cdot)$ is a step function.
\item If $\alpha_n t_{\ent}^{(n)} \rightarrow \infty$, $d_{\alpha_n,\sigma_n}(s  / \alpha_n) \xrightarrow[n \rightarrow \infty]{\P} e^{-s}\ \forall\ s > 0$, i.e.\ $d_{\alpha_n,\sigma_n}(t)$ decays exponentially in $t$.
\item If $\alpha_n t_{\ent}^{(n)} \rightarrow (0,\infty)$, the behavior is intermediate: for $t < t_{\ent}^{(n)}$, $d_{\alpha_n,\sigma_n}(t)$ decays exponentially, as in the $\alpha_n t_{\ent}^{(n)} \rightarrow \infty$ case; for $t > t_{\ent}^{(n)}$, $d_{\alpha_n,\sigma_n}(t) = 0$, as in the $\alpha_n t_{\ent}^{(n)} \rightarrow 0$ case.
\end{itemize}
In \cite{avena2018random}, the authors study a dynamic version of the DCM for which an $\alpha_n$ fraction of edges are randomly sampled and re-paired at each time step. The main result (Theorem 1.4) says the distance to stationarity of the non-backtracking random walk on this dynamic DCM follows a trichotomy similar to the one from \cite{caputo2019mixing}. Similar results were established under weaker assumptions in \cite{avena2020linking}. Finally, \cite{vial2019structural}, also using ideas from \cite{bordenave2016cutoff,bordenave2018random}, studies the matrix $\Pi_n$ with rows $\{ \pi_{\alpha_n,e_i} \}_{i \in [n]}$, i.e.\ the $i$-th row corresponds to restarting at node $i$. The authors study
\begin{align}
\mathcal{D}_n(\epsilon) & = \min_{K_n \subset [n]} \Big( | K_n | \\
& \quad + \left| \left\{ i \notin K_n : \inf_{ \{ \beta_v(k) \}_{k \in K_n} \subset \R  } \left\| \pi_{\alpha_n,e_i} - \left( \alpha_n e_i + {\textstyle \sum}_{k \in K_n} \beta_v(k) \pi_{\alpha_n,e_k} \right) \right\|_1 \geq \epsilon \right\} \right| \Big),
\end{align}
which can be viewed as a measure of the dimension of $\Pi_n$, and whose form is motivated algorithmically (see Section 2.3 in \cite{vial2019structural}). When $P_n$ describes the random walk on the DCM and $\alpha_n = c / t_{\ent}^{(n)}$, the authors show $\mathcal{D}_n(\epsilon) = O ( n^{f(c,\epsilon)} )$, where $f(c,\epsilon) \in (0,1)$ also depends on the given degrees (see Theorem 1 in \cite{vial2019structural}). If instead $\alpha_n t_{\ent}^{(n)} \rightarrow 0$, the authors show $\mathcal{D}_n(\epsilon) = O(1)$, and if $\alpha_n t_{\ent}^{(n)} \rightarrow \infty$, they conjecture $\mathcal{D}_n(\epsilon) = \Omega ( n^s )\ \forall\ s \in (0,1)$, e.g.\ $\mathcal{D}_n(\epsilon) = \Theta ( n / \log n )$ (see Section 7.4 in \cite{vial2019structural}).

Ultimately, as discussed in the introduction, these results all echo Theorem \ref{thmAllRegimes} and hint at a deeper phenomena. However, prior to this work, one may have (erroneously) suspected that such results rely crucially on some property of the DCM, since \cite{caputo2019mixing,avena2018random,vial2019structural} all study this generative model. In contrast, the present paper suggests that some notion of cutoff is the crucial property. Accordingly, it is unsurprising that the trichotomy results in \cite{caputo2019mixing,vial2019structural} rely on the cutoff results from \cite{bordenave2016cutoff,bordenave2018random}.

Our other result, Theorem \ref{thmEquivalence}, relates closely to the aforementioned \cite{basu2015characterization}. Here it is shown that mixing cutoff \eqref{eqDefnCutoff} is equivalent to a certain notion of ``hitting time cutoff''. Namely, Theorem 3 in \cite{basu2015characterization} shows that for sequences of lazy, reversible, and irreducible chains, cutoff \eqref{eqDefnCutoff} is equivalent to each of the following:
\begin{gather} \label{eqDefnHittingCutoff}
\exists\ \eta \in (0,1/2] \textrm{ s.t.\ } t_{\hit}^{(n)}(\eta,\epsilon) - t_{\hit}^{(n)}(\eta,1-\epsilon) = o ( t_{\hit}^{(n)}(\eta,1/4) )\ \forall\ \epsilon \in (0,1/4) , \\
\exists\ \eta \in (1/2,1) \textrm{ s.t.\ } t_{\hit}^{(n)}(\eta,\epsilon) - t_{\hit}^{(n)}(\eta,1-\epsilon) = o ( t_{\hit}^{(n)}(\eta,1/4) )\ \forall\ \epsilon \in (0,1/4), t_{\rel}^{(n)} = o ( t_{\mix}^{(n)} ) .
\end{gather}
Here $t_{\rel}^{(n)}$ is the inverse spectral gap of $P_n$ (see \eqref{eqDefnRelaxTime} in Section \ref{secProofLower})), and $t_{\hit}^{(n)}(\eta,\epsilon)$ is the first time the chain has visited all sets of stationary measure at least $\eta$ with probability at least $1-\epsilon$, from any starting state (see \eqref{eqDefnThit} in Section \ref{secProofLower}). Hence, \eqref{eqDefnHittingCutoff} roughly says that shortly after ``large'' sets are reached at all, they are reached with high probability. As discussed in Section \ref{secIntro}, Theorem \ref{thmEquivalence} nicely complements this result, since both the cutoff notion and the equivalent notion differ.

More broadly, the cutoff phenomena has been established for a variety of chains. In addition to the aforementioned \cite{bordenave2016cutoff,bordenave2018random}, a non-exhaustive list includes Markovian models of card shuffling \cite{bayer1992trailing,diaconis1981generating}, diffusion models \cite{aldous1983random,diaconis1987time}, random walks on regular graphs \cite{lubetzky2010cutoff}, walks on Ramanujan graphs \cite{lubetzky2016cutoff}, non-backtracking walks \cite{ben2017cutoff}, walks on dynamic graphs \cite{sousi2020cutoff}, walks on graphs with communities \cite{ben2020threshold}, walks on Abelian groups \cite{hermon2021cutoff}, and chains with a certain notion of non-negative curvature \cite{salez2021cutoff}. We refer readers to these papers and the references therein, along with the classical survey \cite{diaconis1996cutoff} and the textbooks \cite{aldous2002reversible,levin2009markov}, for further background on mixing times and the cutoff phenomena.

Finally, we mention some prior work with less immediate connections to our own. First, we note that the basic connection between mixing times and perturbation bounds has been previously been explored; for instance, the line of work \cite{mitrophanov2003stability,mitrophanov2005sensitivity} derives upper bounds for perturbation error in terms of mixing times. However, our lower bounds and the precise asymptotic characterization in Theorem \ref{thmAllRegimes} are (to the best of our knowledge) new. In the PageRank literature, another relevant paper is \cite{avrachenkov2015pagerank}, which estimates $\pi_{\alpha_n,\sigma_n}$ as a mixture of $\sigma_n$ and the degree distribution; in this sense, the results in \cite{avrachenkov2015pagerank} are more precise than ours, but they are restricted to a certain class of $P_n$.

\section{Proof of Theorem \ref{thmAllRegimes}} \label{secProofTrichotomy}

In this section, we prove the trichotomy of Theorem \ref{thmAllRegimes}. In particular, the upper and lower bounds of the theorem are presented as Corollaries \ref{corFinalUpper} and \ref{corLowerWithoutDeltas}, respectively, in Sections \ref{secProofUpper} and \ref{secProofLower}, respectively. We note that, throughout Sections \ref{secProofUpper} and \ref{secProofLower}, we derive several intermediate bounds under weaker assumptions than those of the theorem, which may be of independent interest.

\subsection{Upper bounds} \label{secProofUpper}

We begin with a non-asymptotic upper bound on the perturbation error, which requires no assumptions of laziness, reversibility, or cutoff.
 
\begin{lemma} \label{lemGenericUpper}
Let $P_n \in \mathcal{E}_n$ and $\alpha_n \in (0,1)$. Then for any $\tilde{P}_n \in B(P_n,\alpha_n)$ and any $t \in \N$, 
\begin{equation}
\| \pi_n - \tilde{\pi}_n \| \leq 1 - (1-\alpha_n)^t + d_n(t) .
\end{equation}
\end{lemma}

Before proving the lemma, we recall the following standard notions (see Lemma \ref{lemBasicTV} in Appendix \ref{appExisting} for details). A \textit{coupling} of $\mu, \nu \in  \Delta_{n-1}$ is pair of random variables $(X,Y)$ with marginal distributions $\mu$ and $\nu$, $\| \mu - \nu \| = \inf \{ \P ( X \neq Y ) : (X,Y) \text{ is a coupling of } (\mu,\nu) \} $, and the infimum is attained. We call any such $(X,Y)$ attaining the infimum an \textit{optimal coupling}, and we say a coupling \textit{succeeds} if $X=Y$. Hence, an optimal coupling of $(\mu,\nu)$ succeeds with probability $1-\| \mu - \nu \|$.

To prove Lemma \ref{lemGenericUpper}, we first use the triangle inequality, convexity, and the definition of the distance to stationarity function \eqref{eqDefnDnT} to obtain
\begin{equation} \label{lemGenericUpperTriangle}
\| \pi_n - \tilde{\pi}_n \| \leq \| \pi_n - \tilde{\pi}_n P_n^t \| + \| \tilde{\pi}_n P_n^t - \tilde{\pi}_n \| \leq d_n(t) + \| \tilde{\pi}_n P_n^t - \tilde{\pi}_n \| .
\end{equation}
Hence, it suffices to show $\| \tilde{\pi}_n P_n^t - \tilde{\pi}_n \| = \| \tilde{\pi}_n P_n^t - \tilde{\pi}_n \tilde{P}_n^t \| \leq 1 - (1-\alpha_n)^t$. To do so, we use a sequence of $t$ optimal couplings to construct a (possibly suboptimal) coupling of $(\tilde{\pi}_n P_n^t , \tilde{\pi}_n \tilde{P}_n^t)$. Our construction ensures that the coupling of $(\tilde{\pi}_n P_n^t , \tilde{\pi}_n \tilde{P}_n^t)$ succeeds if all $t$ optimal couplings succeed, and the definition of $B(P_n,\alpha_n)$ ensures that each optimal coupling succeeds with probability at least $1-\alpha_n$. Hence, $(\tilde{\pi}_n P_n^t , \tilde{\pi}_n \tilde{P}_n^t)$ succeeds with probability at least $(1-\alpha_n)^t$, which implies the desired bound. This is made precise next.

\begin{proof}[Proof of Lemma \ref{lemGenericUpper}]
As discussed above, we construct a coupling to show $\| \tilde{\pi}_n P_n^t - \tilde{\pi}_n \| \leq 1 - (1-\alpha_n)^t$. Namely, we inductively define a sequence of random variables $\{ X_n(\tau), \tilde{X}_n(\tau) \}_{\tau=0}^t$ as follows:
\begin{itemize}
\item For $\tau = 0$, sample $X_n(\tau)$ from $\tilde{\pi}_n$, and set $\tilde{X}_n(\tau) = X_n(\tau)$.
\item For $\tau \in \{1,\ldots,t\}$, consider two cases:
\begin{enumerate}
\item If $X_n(\tau-1) \neq \tilde{X}_n(\tau-1)$, independently sample $X_n(\tau)$ from $e_{X_n(\tau-1)} P_n$ and $\tilde{X}_n(\tau)$ from $e_{ \tilde{X}_n(\tau-1) } \tilde{P}_n$, respectively. \label{caseTrichotUpperCouplingNotEqual}
\item If $X_n(\tau-1) = \tilde{X}_n(\tau-1)$, let $(X_n(\tau),\tilde{X}_n(\tau))$ be any optimal coupling of $e_{X_n(\tau-1)} P_n$ and $e_{\tilde{X}_n(\tau-1)} \tilde{P}_n = e_{X_n(\tau-1)} \tilde{P}_n$. \label{caseTrichotUpperCouplingEqual}
\end{enumerate}
\end{itemize}
We claim $X_n(\tau) \sim \tilde{\pi}_n P_n^{\tau} , \tilde{X}_n(\tau) \sim \tilde{\pi}_n\ \forall\ \tau \in \{0,\ldots,\tau\}$, where $Z \sim \mu$ means the marginal distribution of $Z$ is $\mu$. Note that, since $\tilde{\pi}_n$ is the stationary distribution of $\tilde{P}_n$, this claim is equivalent to $X_n(\tau) \sim \tilde{\pi}_n P_n^{\tau} , \tilde{X}_n(\tau) \sim \tilde{\pi}_n \tilde{P}_n^{\tau}\ \forall\ \tau$. For $\tau = 0$ and Case \ref{caseTrichotUpperCouplingNotEqual} of $\tau \in \{1,\ldots,t\}$, this latter claim clearly holds inductively by the construction above; for Case \ref{caseTrichotUpperCouplingEqual} of $\tau \in \{1,\ldots,t\}$, it holds inductively and by the coupling definition. Therefore, by definition of optimal coupling,
\begin{equation}\label{lemGenericUpperCoupling}
\| \tilde{\pi}_n P_n^t - \tilde{\pi}_n \| \leq 1 - \P ( X_n(t) = \tilde{X}_n(t) )  .
\end{equation}
We thus seek a lower bound for $\P ( X_n(t) = \tilde{X}_n(t) )$. We first write
\begin{align}
& \P ( X_n(t) = \tilde{X}_n(t) ) =  \E [ \P ( X_n(t) = \tilde{X}_n(t) | X_n(t-1) , \tilde{X}_n(t-1) ) ] \\
&  \quad \geq  \E [ \P ( X_n(t) = \tilde{X}_n(t) | X_n(t-1) , \tilde{X}_n(t-1) ) 1 ( X_n(t-1) = \tilde{X}_n(t-1) ) ] ,
\end{align}
where $1(\cdot)$ is the indicator function. Now conditioned on $X_n(t-1) , \tilde{X}_n(t-1)$, the definition of an optimal coupling guarantees that, whenever $X_n(t-1) = \tilde{X}_n(t-1)$,
\begin{equation}
\P ( X_n(t) = \tilde{X}_n(t) | X_n(t-1) , \tilde{X}_n(t-1) ) = 1 - \| e_{X_n(\tau-1)} P_n - e_{X_n(\tau-1)} \tilde{P}_n \|  \geq 1 - \alpha_n ,
\end{equation}
where the inequality holds pointwise by assumption $\tilde{P}_n \in B(P_n,\alpha_n)$. Using the previous two inequalities, and since $X_n(0) = \tilde{X}_n(0)$ by construction, we then obtain
\begin{equation}\label{lemGenericUpperFin}
\P ( X_n(t) = \tilde{X}_n(t) ) \geq (1-\alpha_n) \P ( X_n(t-1) = \tilde{X}_n(t-1) ) \geq \cdots \geq (1-\alpha_n)^t .
\end{equation}
Combining \eqref{lemGenericUpperTriangle}, \eqref{lemGenericUpperCoupling}, and \eqref{lemGenericUpperFin} yields the desired inequality.
\end{proof}

Using Lemma \ref{lemGenericUpper}, we can obtain the following asymptotic bounds, which again require no assumptions of laziness, reversibility, or cutoff.
\begin{corollary} \label{corUpperWeakAss}
Let $P_n \in \mathcal{E}_n, \alpha_n \in (0,1)\ \forall\ n \in \N$, and let $\epsilon \in (0,1)$ be independent of $n$. Assume $\lim_{n \rightarrow \infty} \alpha_n t_{\mix}^{(n)}(\epsilon) = c \in [ 0,\infty )$. Then the following hold:
\begin{itemize}
\item  If $c = 0$ and $\epsilon < 1/2$, then $\forall\ \{ \tilde{P}_n \}_{n \in \N}$ s.t.\ $\tilde{P}_n \in  B(P_n,\alpha_n)\ \forall\ n \in \N$,
\begin{equation} \label{eqZeroRegime}
\lim_{n \rightarrow \infty} \| \pi_n - \tilde{\pi}_n \| = 0 .
\end{equation}
\item  If $c \in (0,\infty)$ and $\lim_{n \rightarrow \infty} t_{\mix}^{(n)}(\epsilon) = \infty$, then $\forall\ \{ \tilde{P}_n \}_{n \in \N}$ s.t.\ $\tilde{P}_n \in  B(P_n,\alpha_n)\ \forall\ n \in \N$,
\begin{equation} \label{eqMiddleRegimeUpper}
\limsup_{n \rightarrow \infty} \| \pi_n - \tilde{\pi}_n \| \leq \min \{ 1 - e^{-c} + \epsilon , 1 \} .
\end{equation}
\end{itemize}
\end{corollary}
\begin{proof}
To prove \eqref{eqZeroRegime}, we first use Lemma \ref{lemGenericUpper} and Bernoulli's inequality to write
\begin{equation}
\| \pi_n - \tilde{\pi}_n \| \leq 1 - (1-\alpha_n t) + d_n(t) = \alpha_n t + d_n(t) .
\end{equation}
Setting $k_n = \ceil*{ 1 / \sqrt{ \alpha_n  t_{\mix}^{(n)}(\epsilon) } }$, choosing $t  = k_n t_{\mix}^{(n)}(\epsilon)$, and using Lemma \ref{lemBasicMixing} from Appendix \ref{appExisting} (an elementary result for mixing times) then gives
\begin{align}
\| \pi_n - \tilde{\pi}_n \| & \leq k_n \alpha_n t_{\mix}^{(n)}(\epsilon) + (2 \epsilon)^{k_n}  
\xrightarrow[n \rightarrow \infty]{} 0 ,
\end{align}
where the limit holds since $\alpha_n t_{\mix}^{(n)}(\epsilon) \rightarrow 0, \epsilon < 1/2$ by assumption.

To prove \eqref{eqMiddleRegimeUpper}, we simply choose $t = t_{\mix}^{(n)}(\epsilon)$ in Lemma \ref{lemGenericUpper} and use \eqref{eqDefnTmix} to obtain
\begin{equation}
\| \pi_n - \tilde{\pi}_n \| \leq 1 - ( 1 - \alpha_n )^{ t_{\mix}^{(n)}(\epsilon) } + \epsilon \xrightarrow[n \rightarrow \infty]{} 1 - e^{-c} + \epsilon ,
\end{equation}
where the limit holds since $\alpha_n t_{\mix}^{(n)}(\epsilon) \rightarrow c$ by assumption. (In the case $c > \log(1/\epsilon)$, the bound $\| \pi_n - \tilde{\pi}_n \| \leq 1$ in \eqref{eqMiddleRegimeUpper} is obvious.)
\end{proof}

Finally, we prove the upper bounds of Theorem \ref{thmAllRegimes}, which use the assumptions of laziness and cutoff to strenghten the results of the previous corollary. Roughly speaking, the cutoff assumption means that all mixing times are equivalent, which allows us to discard the $\epsilon < 1/2$ assumption when $c = 0$, and to take $\epsilon$ arbitrarily small in \eqref{eqMiddleRegimeUpper}. Furthermore, the laziness assumption allows us to discard the assumption $t_{\mix}^{(n)}(\epsilon) \rightarrow \infty$ when $c \in (0,\infty)$. The proof is fairly straightforward so is deferred to the appendix.

\begin{corollary} \label{corFinalUpper}
Let $P_n \in \mathcal{E}_n, \alpha_n \in (0,1)\ \forall\ n \in \N$, and let $\epsilon \in (0,1)$ be independent of $n$. Assume $\{ P_n \}_{n \in \N}$ exhibits cutoff, each $P_n$ is lazy, and $\lim_{n \rightarrow \infty} \alpha_n t_{\mix}^{(n)}(\epsilon) = c \in [0,\infty]$. Then the following hold:
\begin{itemize} 
\item If $c = 0$, then $\forall\ \{ \tilde{P}_n \}_{n \in \N}$ s.t.\ $\tilde{P}_n \in B(P_n,\alpha_n)\ \forall\ n \in \N$,
\begin{equation} \label{eqCorZeroRegime}
\lim_{n \rightarrow \infty} \| \pi_n - \tilde{\pi}_n \| = 0  .
\end{equation}
\item If $c \in (0,\infty)$, then $\forall\ \{ \tilde{P}_n \}_{n \in \N}$ s.t.\ $\tilde{P}_n \in B(P_n,\alpha_n)\ \forall\ n \in \N$,
\begin{equation} \label{eqCorMiddleRegimeUpper}
\limsup_{n \rightarrow \infty} \| \pi_n - \tilde{\pi}_n \| \leq 1 - e^{-c}  .
\end{equation}
\end{itemize}
\end{corollary}
\begin{proof}
See Appendix \ref{appProofFinalUpper}.
\end{proof}

\subsection{Lower bounds} \label{secProofLower}

We begin with some definitions. First, denote the \textit{hitting time} of $A \subset [n]$ by $T_n(A) = \inf \{ t \in \Z_+ : X_n(t) \in A \}$. Given $\eta_1, \eta_3 \in (0,1)$, we also define
\begin{equation} \label{eqDefnThit}
t_{\hit}^{(n)} (1-\eta_3,\eta_1) = \min \left\{ t : \max_{x \in [n] , A \subset [n] : \pi_n(A) \geq 1 - \eta_3} \P_x ( T_n(A) > t ) \leq \eta_1 \right\} ,
\end{equation}
where $\P_x$ denotes probability conditioned on the chain starting from $X_n(0) = x$. Thus, in words, $t_{\hit}^{(n)} (1-\eta_3,\eta_1)$ is the first time at which the chain has visited all sets of stationary measure at least $1-\eta_3$ from any starting state with probability at least $1-\eta_1$. 

The first key idea of our lower bounds is that we can use this definition to construct a restart perturbation with large perturbation error. To illustrate this, we choose $\eta_3 \approx 0$ and $\eta_1 \approx 1$ and use the definition of $ t_{\hit}^{(n)}$ to find a state $x_n$ and a set of states $A_n$ such that $\pi_n(A_n) \approx 1$ and $\P_{x_n}  ( X_n(t) \in A_n ) \approx 0\ \forall\ t < t_{\hit}^{(n)} (1-\eta_3,\eta_1)$. In particular, if $1 / \alpha_n < t_{\hit}^{(n)} (1-\eta_3,\eta_1)$, we can find $x_n$ and $A_n$ such that $\P_{x_n}  ( X_n(t) \in A_n ) \approx 0\ \forall\ t < 1/\alpha_n$. Hence, if we perturb the chain by restarting at $x_n$ with probability $\alpha_n$ at each step, the perturbed chain will rarely visit $A_n$ (since the number of steps between restarts is $1 / \alpha_n$ in expectation). This implies $\tilde{\pi}_n(A_n) \approx 0$, and since $\pi_n(A_n) \approx 1$ by definition, we obtain $\| \pi_n - \tilde{\pi}_n \| \approx 1$.

More generally, the following lemma shows we can find a perturbation such that the product $\alpha_n  t_{\hit}^{(n)} (1-\eta_3,\eta_1)$ controls the perturbation error $\| \pi_n - \tilde{\pi}_n \|$. We note this lower bound does not require laziness, reversibility, or cutoff assumptions.

\begin{lemma} \label{lemLowerWeaker}
Let $P_n \in \mathcal{E}_n$, $\alpha_n \in (0,1)$, and $\delta \in (0,1/2)$. Then $\exists\ \tilde{P}_n \in B(P_n,\alpha_n)$ s.t.\
\begin{equation}\label{lemLowerWeakerRes}
\| \pi_n - \tilde{\pi}_n \| \geq 1 - 3 \delta - \exp \left( - \alpha_n t_{\hit}^{(n)} (1-\delta,1-2\delta)  \right) .
\end{equation}
\end{lemma}
\begin{proof}
First note $1-\delta,1-2\delta \in (0,1)$ by assumption on $\delta$, so $t_n := t_{\hit}^{(n)} (1-\delta,1-2\delta)$ is well-defined. Hence, by definition, $\exists\ x_n \in [n], A_n \subset [n]$ satisfying
\begin{equation} \label{eqInfWeakXnAn}
\pi_n(A_n) \geq 1 - \delta , \quad \P_{x_n} ( T_n(A_n) > t_n - 1 ) > 1 - 2 \delta .
\end{equation}
Now set $\tilde{P}_n = P_{\alpha_n, e_{x_n} }$ (i.e.\ restart at $x_n$ with probability $\alpha_n$). Then by Lemma \ref{lemPropRestart} in Appendix \ref{appExisting} (a well-known formula for the stationary distribution of a restart perturbation),
\begin{equation} \label{eqInfWeakSum}
\tilde{\pi}_n(A_n) = \alpha_n \sum_{t=0}^{t_n - 1 } (1-\alpha_n)^t \P_{x_n} ( X_n(t) \in A_n ) + \alpha_n \sum_{t=t_n  }^{\infty} (1-\alpha_n)^t \P_{x_n} ( X_n(t) \in A_n ) .
\end{equation}
We consider the two summands in \eqref{eqInfWeakSum} in turn. For the first summand, we note
\begin{equation}
\P_{x_n} ( X_n(t) \in A_n )  \leq \P_{x_n} ( T_n(A_n) \leq t ) \leq \P_{x_n} ( T_n(A_n) \leq t_n - 1 ) < 2 \delta ,
\end{equation}
where the second inequality holds for $t < t_n$, and the third holds by \eqref{eqInfWeakXnAn}. It follows that
\begin{equation}
\alpha_n \sum_{t=0}^{t_n - 1 } (1-\alpha_n)^t \P_{x_n} ( X_n(t) \in A_n ) < 2 \delta .
\end{equation}
For the second summand in \eqref{eqInfWeakSum}, we simply upper bound the probabilities by $1$ to obtain
\begin{equation}
\alpha_n \sum_{t=t_n }^{\infty} (1-\alpha_n)^t \P_{x_n} ( X_n(t) \in A_n ) \leq (1-\alpha_n)^{ t_n } \leq \exp(- \alpha_n t_n ) .
\end{equation}
Taken together, we have shown $\tilde{\pi}_n(A_n) < 2 \delta + \exp( -\alpha_n t_n )$. Combined with \eqref{eqInfWeakXnAn}, we obtain
\begin{equation} \label{eqInfWeakBeforePeres}
\| \pi_n - \tilde{\pi}_n \| \geq \pi_n(A_n) - \tilde{\pi}_n(A_n) > 1 - 3 \delta - \exp( -\alpha_n t_n ) ,
\end{equation}
which completes the proof.
\end{proof}

We have thus identified a restart perturbation for which perturbation error is lower bounded in terms of $\alpha_n  t_{\hit}^{(n)} (1-\eta_3,\eta_1)$. The second key idea of our lower bounds, taken from the aforementioned \cite{basu2015characterization}, is that $t_{\hit}^{(n)} (1-\eta_3,\eta_1)$ can be lower bounded in terms of $t_{\mix}^{(n)} (\epsilon)$. To formulate this, we first define the \textit{relaxation time}
\begin{equation} \label{eqDefnRelaxTime}
t_{\rel}^{(n)}  = \frac{1}{ 1-\lambda_n^* }  ,
\end{equation}
where $1-\lambda_n^*$ is the absolute spectral gap of $P_n$, defined by
\begin{equation} \label{eqDefnSpectralGap}
\lambda_n^* = \max \{ |\lambda| : \lambda \textrm{ is an eigenvalue of } P_n, \lambda \neq 1 \} .
\end{equation}
(Note $P_n \in \mathcal{E}_n \Rightarrow \lambda_n^* < 1$ -- see e.g.\ Lemma 12.1 in \cite{levin2009markov} -- so \eqref{eqDefnRelaxTime} is well-defined in this case.) We then have the following.
\begin{lemma}[Corollary 3.1 in \cite{basu2015characterization}] \label{lemHitIneq}
Let $P_n \in \mathcal{E}_n$ and assume $P_n$ is lazy and reversible. Then for any $\eta_1,\eta_2,\eta_3 \in (0,1)$,
\begin{equation}
t_{\mix}^{(n)} ( (\eta_1 + \eta_2) \wedge 1 ) \leq t_{\hit}^{(n)} (1-\eta_3,\eta_1) + \ceil*{ \frac{t_{\rel}^{(n)}}{2} \max \left\{ \log \left( \frac{2 (1-\eta_1)^2}{\eta_1 \eta_2 \eta_3} \right) , 0 \right\} } .
\end{equation}
\end{lemma}

Combining the previous two lemmas, we obtain lower bounds of the form
\begin{equation} \label{eqLowerWithoutCutoff}
\| \pi_n - \tilde{\pi}_n \| \gtrapprox 1 - \exp \left( - \alpha_n \left(   t_{\mix}^{(n)}  -  t_{\rel}^{(n)}  \right) \right)  .
\end{equation}
This is the form we desire, except for the term $t_{\rel}^{(n)}$. However, it is known that $t_{\rel}^{(n)} = o ( t_{\mix}^{(n)}(\epsilon) )$ when pre-cutoff holds, so we can obtain nontrivial lower bounds using (pre-)cutoff assumptions. In particular, the next corollary considers the cases when cutoff holds and $\alpha_n t_{\mix}^{(n)}(\epsilon) \rightarrow (0,\infty)$, or pre-cutoff holds and $\alpha_n t_{\mix}^{(n)}(\epsilon) \rightarrow \infty$. Hence, a sharper convergence to stationarity requires a smaller perturbation magnitude (i.e.\ $\alpha_n t_{\mix}^{(n)}(\epsilon) \rightarrow (0,\infty)$ instead of $\alpha_n t_{\mix}^{(n)}(\epsilon) \rightarrow \infty$) to obtain a useful lower bound on $\| \pi_n - \tilde{\pi}_n \|$.

\begin{corollary} \label{corLowerWithDeltas}
Let $P_n \in \mathcal{E}_n, \alpha_n \in (0,1)\ \forall\ n \in \N$, and let $\delta \in (0,1/2)$ be independent of $n$. Assume each $P_n$ is lazy and reversible and $\lim_{n \rightarrow \infty} \alpha_n t_{\mix}^{(n)}(\epsilon) = c \in (0,\infty]$.
\begin{itemize}
\item If $c = \infty$ and $\{ P_n \}_{n \in \N}$ has pre-cutoff, then $\exists\ \{ \tilde{P}_n \}_{n \in \N}$ s.t.\ $\tilde{P}_n \in B(P_n,\alpha_n)\ \forall\ n \in \N$ and
\begin{equation}\label{lemLowerWeakerInf}
\liminf_{n \rightarrow \infty} \| \pi_n - \tilde{\pi}_n \| \geq 1 - 3 \delta .
\end{equation}
\item If $c \in (0,\infty)$ and $\{ P_n \}_{n \in \N}$ has cutoff, then $\exists\ \{ \tilde{P}_n \}_{n \in \N}$ s.t.\ $\tilde{P}_n \in B(P_n,\alpha_n)\ \forall\ n \in \N$ and
\begin{equation}\label{lemLowerWeakerC}
\liminf_{n \rightarrow \infty} \| \pi_n - \tilde{\pi}_n \| \geq 1 - 3 \delta - e^{-c} .
\end{equation}
\end{itemize}
\end{corollary}
\begin{proof}
We first prove \eqref{lemLowerWeakerInf}. Note the assumptions of Lemma \ref{lemLowerWeaker} are satisfied for each $n \in \N$, so we can find $\{ \tilde{P}_n \}_{n \in \N}$ with $\tilde{P}_n \in B(P_n,\alpha_n)$ for each $n$, and
\begin{align}
& \liminf_{n \rightarrow \infty} \| \pi_n - \tilde{\pi}_n \| \\
& \quad \geq 1 - 3 \delta - \exp \left( - \liminf_{n \rightarrow \infty}  \alpha_n t_{\mix}^{(n)}(\delta) \left( \frac{t_{\mix}^{(n)}(1-\delta)}{t_{\mix}^{(n)}(\delta)} - \ceil*{ \frac{t_{\rel}^{(n)}}{2 t_{\mix}^{(n)}(\delta)} \log \left( \frac{ 8 }{ 1 - 2 \delta} \right) } \right) \right) .
\end{align}
Now when pre-cutoff holds, $t_{\mix}^{(n)}(1-\delta)  / t_{\mix}^{(n)}(\delta)$ is lower bounded by a positive constant as $n \rightarrow \infty$ (by definition) and $t_{\rel}^{(n)} = o ( t_{\mix}^{(n)}(\delta) )$ (by Lemma \ref{lemBasicMixing} in Appendix \ref{appExisting}). Thus, if $\alpha_n t_{\mix}^{(n)}(\delta) \rightarrow \infty$, we obtain
\begin{equation}
\lim_{n \rightarrow \infty}  \alpha_n t_{\mix}^{(n)}(\delta) \left( \frac{t_{\mix}^{(n)}(1-\delta)}{t_{\mix}^{(n)}(\delta)} - \ceil*{ \frac{t_{\rel}^{(n)}}{2 t_{\mix}^{(n)}(\delta)} \log \left( \frac{ 8 }{ 1 - 2 \delta} \right) } \right) = \infty ,
\end{equation}
and combining the previous two lines yields \eqref{lemLowerWeakerInf}.

The proof of \eqref{lemLowerWeakerC} is nearly identical, except $t_{\mix}^{(n)}(1-\delta)  / t_{\mix}^{(n)}(\delta) \rightarrow 1$ when cutoff holds, so if $\alpha_n t_{\mix}^{(n)}(\delta) \rightarrow c \in (0,\infty)$, the limit in the previous equation equals $c$.
\end{proof}

The bounds in Corollary \ref{corLowerWithDeltas} match the lower bounds of Theorem \ref{thmAllRegimes} except for the $3 \delta$ terms. However, these are largely ``nuisance terms'': roughly speaking, when cutoff holds, all $\delta$-mixing times are equivalent, so one can apply Corollary \ref{corLowerWithDeltas} with arbitrarily small $\delta$ ensure these nuisance terms vanish. This is formalized in the following corollary. The proof is delicate but uses elementary tools so is deferred to the appendix.

\begin{corollary} \label{corLowerWithoutDeltas}
Let $P_n \in \mathcal{E}_n, \alpha_n \in (0,1)\ \forall\ n \in \N$, and let $\epsilon \in (0,1)$ be independent of $n$. Assume each $P_n$ is lazy and reversible and $\lim_{n \rightarrow \infty} \alpha_n t_{\mix}^{(n)}(\epsilon) = c \in (0,\infty]$.
\begin{itemize}
\item If $c = \infty$ and $\{ P_n \}_{n \in \N}$ has pre-cutoff, then $\exists\ \{ \tilde{P}_n \}_{n \in \N}$ s.t.\ $\tilde{P}_n \in B(P_n,\alpha_n)\ \forall\ n \in \N$ and
\begin{equation}\label{lemLowerStongerInf}
\lim_{n \rightarrow \infty} \| \pi_n - \tilde{\pi}_n \| = 1 .
\end{equation}
\item If $c \in (0,\infty)$ and $\{ P_n \}_{n \in \N}$ has cutoff, then $\exists\ \{ \tilde{P}_n \}_{n \in \N}$ s.t.\ $\tilde{P}_n \in B(P_n,\alpha_n)\ \forall\ n \in \N$ and
\begin{equation}\label{lemLowerStongerC}
\liminf_{n \rightarrow \infty} \| \pi_n - \tilde{\pi}_n \| \geq 1 - e^{-c} .
\end{equation}
\end{itemize}
\end{corollary}
\begin{proof}
See Appendix \ref{appProofLowerWithoutDeltas}.
\end{proof}

\section{Proof of Theorem \ref{thmEquivalence}} \label{secProofEquivalence}

For convenience, we restate a definition from Section \ref{secResults}: a sequence $\{ \alpha_{n,\epsilon} \}_{n \in \N, \epsilon \in (0,1/2)} \subset (0,1)$ \textit{coincides with} $\{ t_{\mix}^{(n)}(\epsilon) \}_{n \in \N, \epsilon \in (0,1)}$ if
\begin{gather} \label{eqDefnCoincidesRestate}
\sup_{\epsilon \in (0,1/2)} \liminf_{n \rightarrow \infty} \alpha_{n,\epsilon} t_{\mix}^{(n)} (\epsilon) = \infty , \\ \frac{ \alpha_{n,\epsilon} }{ \alpha_{n,\delta} } \in \left[ \frac{ t_{\mix}^{(n)}(1-\delta) }{ t_{\mix}^{(n)}(1-\epsilon) } , 1 \right]\ \forall\ \epsilon, \delta \in (0,1/2) \textrm{ s.t.\ } \epsilon \geq \delta, \forall\ n \in \N .
\end{gather}

Note that such sequences always exist (at least in the case of laziness). In particular, if $c \in (0,1)$ is independent of $n$ and $\epsilon$, and if $\alpha_{n,\epsilon} = c\ \forall\ n \in \N, \epsilon \in (0,1/2)$, then $\{ \alpha_{n,\epsilon} \}_{n \in \N, \epsilon \in (0,1/2)}$ satisfies \eqref{eqDefnCoincidesRestate}. To see why, note that the first condition in \eqref{eqDefnCoincidesRestate} follows immediately from Lemma \ref{lemBasicMixing} in Appendix \ref{appExisting} (a basic result for mixing times); for the second condition, the upper bound is clearly satisfied; also, the interval in \eqref{eqDefnCoincidesRestate} is nonempty by \eqref{eqMixMonotone}. 

Next, we prove the following property that was discussed in Section \ref{secResults}.
\begin{lemma} \label{lemCoincidesProperty}
If pre-cutoff holds and $\{ \alpha_{n,\epsilon} \}_{n \in \N, \epsilon \in (0,1/2)} \subset (0,1)$ coincides with the mixing times $\{ t_{\mix}^{(n)}(\epsilon) \}_{n \in \N, \epsilon \in (0,1)}$, then $\forall\ \epsilon \in (0,1/2)$, $\lim_{n \rightarrow \infty} \alpha_{n,\epsilon} t_{\mix}^{(n)} (\epsilon) = \infty$.
\end{lemma}
\begin{proof}
Let $\epsilon \in (0,1/2)$; we aim to show $\alpha_{n,\epsilon} t_{\mix}^{(n)} (\epsilon) \rightarrow \infty$. Fix $n \in \N$. Then $\forall\ \delta \in (0,\epsilon]$,
\begin{equation} \label{eqCoinPropLB1}
\alpha_{n,\epsilon} t_{\mix}^{(n)} (\epsilon) \geq \alpha_{n,\epsilon} t_{\mix}^{(n)} (1-\epsilon) \geq \alpha_{n,\delta} t_{\mix}^{(n)} (1-\delta)  = \alpha_{n,\delta} t_{\mix}^{(n)} (\delta) \frac{ t_{\mix}^{(n)} (1-\delta) }{ t_{\mix}^{(n)} (\delta) } ,
\end{equation}
where the first inequality holds by \eqref{eqMixMonotone} (since $\epsilon < 1/2$), and the second holds by the lower bound of the interval in \eqref{eqDefnCoincidesRestate}. On the other hand, $\forall\ \delta \in [\epsilon,1/2)$,
\begin{equation} \label{eqCoinPropLB2}
\alpha_{n,\epsilon} t_{\mix}^{(n)} (\epsilon) \geq \alpha_{n,\delta} t_{\mix}^{(n)} (\delta) \geq \alpha_{n,\delta} t_{\mix}^{(n)} (\delta) \frac{ t_{\mix}^{(n)} (1-\delta) }{ t_{\mix}^{(n)} (\delta) } ,
\end{equation}
where the first inequality holds by the upper bound of the interval in \eqref{eqDefnCoincidesRestate} and by \eqref{eqMixMonotone}, and the second holds by \eqref{eqMixMonotone} (since $\delta < 1/2$). Now since $n \in \N$ was arbitrary, \eqref{eqCoinPropLB1} and \eqref{eqCoinPropLB2} imply
\begin{equation}
\liminf_{n \rightarrow \infty} \alpha_{n,\epsilon} t_{\mix}^{(n)} (\epsilon) \geq \liminf_{n \rightarrow \infty} \alpha_{n,\delta} t_{\mix}^{(n)} (\delta) \frac{ t_{\mix}^{(n)} (1-\delta) }{ t_{\mix}^{(n)} (\delta) }\ \forall\ \delta \in (0,1/2) .
\end{equation}
Also, by definition of $\liminf$ and pre-cutoff, $\exists\ K > 0$ independent of $n,\delta$ such that $\forall\ \delta \in (0,1/2)$,
\begin{align}
\liminf_{n \rightarrow \infty} \alpha_{n,\delta} t_{\mix}^{(n)} (\delta) \frac{ t_{\mix}^{(n)} (1-\delta) }{ t_{\mix}^{(n)} (\delta) } & \geq   \liminf_{n \rightarrow \infty} \frac{ t_{\mix}^{(n)} (1-\delta) }{ t_{\mix}^{(n)} (\delta) }   \liminf_{n \rightarrow \infty} \alpha_{n,\delta} t_{\mix}^{(n)} (\delta)   
\\
& \geq K \liminf_{n \rightarrow \infty} \alpha_{n,\delta} t_{\mix}^{(n)} (\delta) . 
\end{align}
Combining the previous two bounds, and since these bounds hold $\forall\ \delta \in (0,1/2)$,
\begin{equation}
\liminf_{n \rightarrow \infty} \alpha_{n,\epsilon} t_{\mix}^{(n)} (\epsilon) \geq K \sup_{\delta \in (0,1/2) } \liminf_{n \rightarrow \infty} \alpha_{n,\delta} t_{\mix}^{(n)} (\delta) = \infty ,
\end{equation}
where the equality holds by \eqref{eqDefnCoincidesRestate}.
\end{proof}

We turn to the proof of the theorem. First, we show pre-cutoff implies Condition \ref{condSens}. For this, let $\{ \alpha_{n,\epsilon} \}_{n \in \N, \epsilon \in (0,1/2)} \subset (0,1)$ coincide with $\{ t_{\mix}^{(n)}(\epsilon) \}_{n \in \N, \epsilon \in (0,1)}$, and fix $\epsilon \in (0,1/2)$. Lemma \ref{lemCoincidesProperty} ensures $\alpha_{n,\epsilon} t_{\mix}^{(n)}(\epsilon) \rightarrow \infty$; hence, by Corollary \ref{corLowerWithoutDeltas}, $\exists\ \{ \sigma_{n,\epsilon} \}_{n \in \N}$ s.t.
\begin{equation}
\sigma_{n,\epsilon} \in \Delta_{n-1}\ \forall\ n \in \N,  \quad  \lim_{n \rightarrow \infty} \| \pi_n - \pi_{\alpha_{n,\epsilon},\sigma_{n,\epsilon}} \| = 1 .
\end{equation}
Next, assume \eqref{eqPrecutFailNice} holds. Let $\alpha_{n,\epsilon} = 1 / ( 2 t_{\mix}^{(n)}(1-\epsilon) )\ \forall\ n \in \N, \epsilon \in (0,1/2)$. Then
\begin{equation}
\frac{\alpha_{n,\epsilon}}{\alpha_{n,\delta}} = \frac{ t_{\mix}^{(n)}(1-\delta) }{ t_{\mix}^{(n)}(1-\epsilon) }\ \forall\ \epsilon, \delta \in (0,1/2) .
\end{equation}
Furthermore, since \eqref{eqPrecutFailNice} holds by assumption,
\begin{equation}
\sup_{\epsilon \in (0,1/2)} \liminf_{n \rightarrow \infty} \alpha_{n,\epsilon} t_{\mix}^{(n)}(\epsilon) = \frac{1}{2} \sup_{\epsilon \in (0,1/2)} \liminf_{n \rightarrow \infty} \frac{ t_{\mix}^{(n)} (\epsilon) }{ t_{\mix}^{(n)} (1-\epsilon) } = \infty . 
\end{equation}
The previous two lines show that $\{ \alpha_{n,\epsilon} \}_{n \in \N, \epsilon \in (0,1/2)}$ coincides with $\{ t_{\mix}^{(n)}(\epsilon) \}_{n \in \N, \epsilon \in (0,1)}$. Fixing $\epsilon \in (0,1/2)$ and $\{ \sigma_{n,\epsilon} \}_{n \in \N}$ s.t.\ $\sigma_{n,\epsilon} \in \Delta_{n-1}\ \forall\ n \in \N$, we will show
\begin{equation}
\| \pi_n  - \pi_{\alpha_{n,\epsilon},\sigma_{n,\epsilon}} \| \leq 1 - \frac{\epsilon}{2} ,
\end{equation}
which will imply that Condition \ref{condSens} fails. Toward this end, we first use Lemma \ref{lemPropRestart} (a well-known formula for the stationary distribution of a restart perturbation) to write
\begin{equation}
\pi_n  - \pi_{\alpha_{n,\epsilon},\sigma_{n,\epsilon}} = \alpha_{n,\epsilon} \sum_{t=0}^{\infty} (1-\alpha_{n,\epsilon})^t ( \pi_n - \sigma_{n , \epsilon} P_n^t ) .
\end{equation}
Hence, by convexity and definition of the $t$-step distance from stationarity \eqref{eqDefnDnT},
\begin{equation}
\| \pi_n  - \pi_{\alpha_{n,\epsilon},\sigma_{n,\epsilon}}  \| \leq \alpha_{n,\epsilon} \sum_{t=0}^{\infty} (1-\alpha_{n,\epsilon})^t \| \pi_n - \sigma_{n , \epsilon} P_n^t \| \leq \alpha_{n,\epsilon} \sum_{t=0}^{\infty} (1-\alpha_{n,\epsilon})^t d_n(t) .
\end{equation}
We thus obtain
\begin{align}
\| \pi_n  - \pi_{\alpha_{n,\epsilon},\sigma_{n,\epsilon}} \| & \leq \alpha_{n,\epsilon} \sum_{t=0}^{t_{\mix}^{(n)}(1-\epsilon)-1} (1-\alpha_{n,\epsilon})^t d_n(t) + \alpha_{n,\epsilon} \sum_{t=t_{\mix}^{(n)}(1-\epsilon)}^{\infty} (1-\alpha_{n,\epsilon})^t d_n(t) \\
& \leq \alpha_{n,\epsilon} \sum_{t=0}^{t_{\mix}^{(n)}(1-\epsilon)-1} (1-\alpha_{n,\epsilon})^t + \alpha_{n,\epsilon} \sum_{t=t_{\mix}^{(n)}(1-\epsilon)}^{\infty} (1-\alpha_{n,\epsilon})^t (1-\epsilon) \\
& = 1 - \epsilon (1-\alpha_{n,\epsilon})^{ t_{\mix}^{(n)}(1-\epsilon) } = 1 - \epsilon \left( 1 - \frac{1/2}{ t_{\mix}^{(n)}(1-\epsilon) } \right)^{ t_{\mix}^{(n)}(1-\epsilon) } \leq 1 - \frac{\epsilon}{2} ,
\end{align}
where the final inequality is Bernoulli's.

\section{Mixing for complete graph bijection} \label{secProofExampleMain}

In this section, we show $t_{\mix}^{(n)} = O(n)$ for the CGB defined in \eqref{eqCgbTrans}. We begin with some definitions. First, we let $N(i)$ denote the neighbors of $i \in [n]$ in the underlying graph, i.e.\
\begin{equation} \label{eqNeighbors}
N(i) = \begin{cases} \{ 1 , \ldots, i-1, i+1, \ldots , \frac{n}{2} , i+\frac{n}{2} \} , & n \textrm{ even}, i \leq \frac{n}{2} \\ \{ i-\frac{n}{2},1+\frac{n}{2},\ldots,i-1,i+1,\ldots,n \}   , & n \textrm{ even} , i > \frac{n}{2} \\ \{ 1 , \ldots , i-1,i+1,\ldots, \frac{n-1}{2} , i+\frac{n-1}{2}, n \} , & n \textrm{ odd}, i \leq \frac{n-1}{2} \\ \{ i-\frac{n-1}{2}, 1+\frac{n-1}{2} , \ldots , i-1,i+1,\ldots, n \}  , & n \textrm{ odd} , \frac{n-1}{2} < i < n \\ \{ 1 , \ldots , n-1 \} , & n \textrm{ odd} , i = n \end{cases} .
\end{equation}
As an example, for the $n = 6$ graph in Figure \ref{fig_ex_cgb}, we have
\begin{gather}
N(1) = \{2,3,4\}, \quad N(2) = \{1,3,5\} , \quad N(3) = \{1,2,6\}, \\
N(4) = \{1,5,6\}, \quad N(5) = \{2,4,6\}, \quad N(6) = \{3,4,5\} ,
\end{gather}
while for the $n= 7$ graph in the same figure, we have
\begin{gather}
N(1) = \{2,3,4,7\}, \quad N(2) = \{1,3,5,7\} , \quad N(3) = \{1,2,6,7\}, \\
N(4) = \{1,5,6,7\}, \quad N(5) = \{2,4,6,7\}, \quad N(6) = \{3,4,5,7\} ,\\
N(7) = \{1,2,3,4,5,6\} .
\end{gather}
For $n$ even, we will refer to the sets $\{1,\ldots,n/2\}$ and $\{1+n/2,\ldots,n\}$ as \textit{cliques} (since these sets form complete subgraphs); similarly, for $n$ odd, we will call the sets $\{1,\ldots,(n-1)/2\}$ and $\{1+(n-1)/2,\ldots,n-1\}$ cliques. For example, $\{1,2,3\}, \{4,5,6\}$ are cliques in Figure \ref{fig_ex_cgb}.

To show $t_{\mix}^{(n)} = O(n)$, we construct couplings. More specifically, by Lemmas \ref{lemBasicTV} and \ref{lemBasicMixing} in Appendix \ref{appExisting} (two elementary mixing time results), we aim to bound
\begin{equation} \label{eqExCouplingApproach}
d_n(t) \leq \max_{i,j \in [n]} \| e_i P_n^t - e_j P_n^t \| \leq \max_{i,j \in [n]} \P_{ij} ( X_n(t) \neq Y_n(t) ) ,
\end{equation}
where $\{ X_n(t) \}_{t \in \Z_+}$ and $\{ Y_n(t) \}_{t \in \Z_+}$, respectively, are Markov chains with transition matrix $P_n$ starting from $X_n(0) = i$ and $Y_n(0) = j$, respectively (as denoted by the subscript in $\P_{ij}$). We separately consider the cases where $n$ is odd and $n$ is even.

\subsection{The even case}

When $n$ is even, the basic approach is to first bring the two chains to the same clique, after which time they remain in the same clique forever. Once the chains are in the same clique, we bring them to the same state, after which time they remain in the same state forever. Concretely, given $X_n(t), Y_n(t)$, we assign $X_n(t+1), Y_n(t+1)$ as follows:
\begin{enumerate}[(A)]
\item \label{coupEvenSameState} If $X_n(t) \neq Y_n(t)$, proceed to \ref{coupEvenDiffClique}. Otherwise, let $X_n(t+1) \sim e_{X_n(t)} P_n$ and set $Y_n(t+1) = X_n(t+1)$ (i.e.\ run the chains together).
\item \label{coupEvenDiffClique} If $X_n(t), Y_n(t)$ are in the same clique, proceed to \ref{coupEvenSameClique}. Otherwise, flip an independent fair coin. If heads, sample $X_n(t+1)$ from $N(X_n(t))$ uniformly (i.e.\ move this chain) and set $Y_n(t+1) = Y_n(t)$ (i.e.\ keep this chain lazy). If tails, set $X_n(t+1) = X_n(t)$ (i.e.\ keep this chain lazy) and sample $Y_n(t+1)$ from $N(Y_n(t))$ uniformly (i.e.\ move this chain).
\item \label{coupEvenSameClique} Flip an independent fair coin. If heads, set $X_n(t+1) = X_n(t), Y_n(t+1) = Y_n(t)$ (i.e.\ keep both chains lazy). If tails, roll a three-sided die that lands $1$, $2$, and $3$ with probability $\frac{2}{n}$, $\frac{2}{n}$, and $1-\frac{4}{n}$, respectively, and proceed as follows:
\begin{itemize}
\item If $1$, define $X_n(t+1),Y_n(t+1)$ as follows (i.e.\ move to the other clique):
\begin{equation}
(X_n(t+1),Y_n(t+1)) = \begin{cases} (X_n(t)+n/2, Y_n(t)+n/2) , & X_n(t) \leq n/2 \\ (X_n(t)-n/2,Y_n(t)-n/2) , & X_n(t)  > n/2 \end{cases} .
\end{equation}
\item If $2$, set $X_n(t+1) = Y_n(t), Y_n(t+1) = X_n(t)$ (i.e.\ swap the chains).
\item If $3$, sample $X_n(t+1)$ uniformly from $N(X_n(t)) \setminus \{ Y_n(t) \}$, set $Y_n(t+1) = X_n(t+1)$.
\end{itemize}
\end{enumerate}
We note that in \ref{coupEvenDiffClique}, we move only one chain to ensure the chains do not switch cliques (meaning, e.g.\ $X_n(t) \in \{1,\ldots,n/2\}, Y_n(t) \in \{1+n/2,\ldots,n\}$ and $X_n(t+1) \in \{1+n/2,\ldots,n\}, Y_n(t+1) \in \{1,\ldots,n/2\}$), which prolongs the time for the chains to meet. We also note that when the die is $2$ or $3$ in \ref{coupEvenSameClique}, both chains move within the clique; by swapping the chains when the die is $2$, we can sample uniformly from the clique, excluding the states $X_n(t), Y_n(t)$, when the die is $3$.

To analyze this coupling, first suppose $X_n(0) = i, Y_n(0) = j$ for $i \neq j$ in the same clique. Then $X_n(t) \neq Y_n(t)$ implies that at each $\tau \in \{ 0,\ldots,t-1 \}$, one of the following occur:
\begin{itemize}
\item The coin in \ref{coupEvenSameClique} lands heads, so that both chains are lazy. This occurs with probability $1/2$.
\item The coin in \ref{coupEvenSameClique} lands tails and the die in \ref{coupEvenSameClique} lands $1$ or $2$, so that both chains move, but to different states. This occurs with probability $(1/2) \times (4/n) = 2/n$. 
\end{itemize}
By independence of these coin flips and die rolls, it follows that
\begin{equation} \label{eqCoupEvenSameStart}
\P_{ij} ( X_n(t) \neq Y_n(t) ) \leq \left( \frac{1}{2} + \frac{2}{n} \right)^t .
\end{equation}
Next, suppose $X_n(0) = i, Y_n(0) = j$ for $i \neq j$ in different cliques. Fix $t \in \N,  \tau \in \{1,\ldots,t\}$, and let $E_{\tau}$ denote the event that $X_n(\tau), Y_n(\tau)$ are in the same clique. Then by the Markov property,
\begin{align}
\P_{ij} ( X_n(t) \neq Y_n(t) ) & = \P ( X_n(t) \neq Y_n(t) | E_{\tau} ) \P ( E_{\tau} | X_n(0) = i, Y_n(0) = j ) \\
& \quad\quad + \P ( X_n(t) \neq Y_n(t) | E_{\tau}^C ) \P ( E_{\tau}^C | X_n(0) = i, Y_n(0) = j ) \\
& \leq \P ( X_n(t) \neq Y_n(t) | E_{\tau} ) + \P ( E_{\tau}^C | X_n(0) = i, Y_n(0) = j ) \label{eqCoupEvenDiffStartSummands}
\end{align}
For the first summand in \eqref{eqCoupEvenDiffStartSummands}, we can use the time invariance of the Markov chains and the fact that \eqref{eqCoupEvenSameStart} holds for any $i \neq j$ in the same clique to obtain
\begin{equation}
 \P ( X_n(t) \neq Y_n(t) | E_{\tau} ) = \P (  X_n(t-\tau) \neq Y_n(t-\tau) | E_0 ) \leq \left( \frac{1}{2} + \frac{2}{n} \right)^{t-\tau} 
\end{equation}
For the second summand  in \eqref{eqCoupEvenDiffStartSummands}, note that $E_{\tau}^C$ implies $X_n(\tau'), Y_n(\tau')$ are not in the same clique for each $\tau' \leq \tau$ (since once the chains reach the same clique, they remain in the same clique forever). This in turn implies that at each such $\tau'$, the chain that moves in \ref{coupEvenDiffClique} at step $\tau'$ moves within its current clique. Such moves occur with probability $1-2/n$. Thus, by independence,
\begin{equation}
\P ( E_{\tau}^C | X_n(0) = i, Y_n(0) = j ) \leq \left( 1 - \frac{2}{n} \right)^{\tau} \leq \exp \left( - \frac{ 2 \tau }{ n } \right) .
\end{equation}
To summarize, we have shown that if $X_n(0) = i, Y_n(0) = j$ for $i \neq j$ not in the same clique,
\begin{equation} \label{eqCoupEvenDiffStart}
\P_{ij} ( X_n(t) \neq Y_n(t) ) \leq  \left( \frac{1}{2} + \frac{2}{n} \right)^{t-\tau} + \exp \left( - \frac{ 2 \tau }{ n } \right) .
\end{equation}
Combining \eqref{eqCoupEvenSameStart} and \eqref{eqCoupEvenDiffStart}, we thus obtain for any $t \in \N,  \tau \in \{1,\ldots,t\}$,
\begin{equation} \label{eqCoupEvenAlmostDone}
\max_{i,j \in [n]} \P_{ij} ( X_n(t) \neq Y_n(t) ) \leq \max \left\{  \left( \frac{1}{2} + \frac{2}{n} \right)^t , \left( \frac{1}{2} + \frac{2}{n} \right)^{t-\tau} + \exp \left( - \frac{ 2 \tau }{ n } \right)  \right\} .
\end{equation}
Now it is straightforward to verify that if (for example)
\begin{equation}
n \geq 6 , \quad \tau \geq  \frac{n}{2} \log \left( \frac{2}{\epsilon} \right)  , \quad t \geq \tau + \frac{ \log(2/\epsilon) }{ \log(6/5) } \geq \frac{n}{2} \log \left( \frac{2}{\epsilon} \right) + \frac{ \log(2/\epsilon) }{ \log(6/5) }  ,
\end{equation}
then \eqref{eqCoupEvenAlmostDone} is further bounded by $\epsilon$. Hence, by \eqref{eqExCouplingApproach}, we obtain for some $a_{\epsilon}$ independent of $n$,
\begin{equation} \label{eqExCouplingNevenFinal}
t_{\mix}^{(n)}(\epsilon) \leq \frac{n}{2} \log \left( \frac{2}{\epsilon} \right) + \frac{ \log(2/\epsilon) }{ \log(6/5) } \leq a_{\epsilon} n\ \forall\ n \in \{ 6, 8,\ldots \} .
\end{equation}

\subsection{The odd case}

When $n$ is odd, we could use a similar approach (bring the chains to the same clique, then bring them to the same state), but the auxiliary state $n$ complicates this. Hence, we instead leverage this auxiliary state as follows: we wait until both chains leave state $n$ (if necessary); we then ensure that the next visits to $n$ occur simultaneously (after which point the chains run together indefinitely). More specifically, given $X_n(t), Y_n(t)$, we assign $X_n(t+1), Y_n(t+1)$ as follows:
\begin{enumerate}[(A),start=4]
\item \label{coupOddSameState} If $X_n(t) \neq Y_n(t)$, proceed to \ref{coupOddOneAtAux}. Otherwise, let $X_n(t+1) \sim e_{X_n(t)} P_n, Y_n(t+1) = X_n(t+1)$.
\item \label{coupOddOneAtAux} If $X_n(t) \neq n$ and $Y_n(t) \neq n$, proceed to \ref{coupOddNoneAtAux}. Otherwise, flip an independent fair coin. If heads, sample $X_n(t+1)$ from $N(X_n(t))$ uniformly and set $Y_n(t+1) = Y_n(t)$. If tails, set $X_n(t+1) = X_n(t)$ and sample $Y_n(t+1)$ from $N(Y_n(t))$ uniformly.
\item \label{coupOddNoneAtAux} Roll a die that lands $1$, $2$, and $3$ with probability $\frac{1}{2}$, $\frac{1}{2} - \frac{1}{n+1}$, and $\frac{1}{n+1}$, respectively.
\begin{itemize}
\item If $1$, set $X_n(t+1) = X_n(t), Y_n(t+1) = Y_n(t)$.
\item If $2$, independently and uniformly sample $X_n(t+1)$ and $Y_n(t+1)$ from $N(X_n(t)) \setminus \{n\}$ and $N(Y_n(t)) \setminus \{n\}$, respectively.
\item If $3$, set $X_n(t+1) = Y_n(t+1) = n$.
\end{itemize}
\end{enumerate}

To analyze this coupling, first suppose $X_n(0) = i, Y_n(0) = j$ for some $i,j \in [n] \setminus \{n\}$ s.t.\ $i \neq j$. Then $X_n(t) \neq Y_n(t)$ implies the following, for each $\tau \leq t$:
\begin{itemize}
\item $X_n(\tau) \neq Y_n(\tau)$. (This can be proven by contradiction. Namely, if $X_n(\tau) = Y_n(\tau)$, then $X_n(t) \neq Y_n(t)$ is violated, since the chains run together forever after meeting by \ref{coupOddSameState}.)
\item $X_n(\tau) \neq n, Y_n(\tau) \neq n$. (This can be proven inductively. For $\tau = 0$, it holds by assumption. For $\tau > 0$, we have $X_n(\tau-1) \neq Y_n(\tau-1)$ by the previous item and $X_n(\tau-1) \neq n, Y_n(\tau-1) \neq n$ by the inductive hypothesis. Hence, $X_n(\tau), Y_n(\tau)$ are assigned via \ref{coupOddNoneAtAux}. This implies $X_n(\tau) \neq n, Y_n(\tau) \neq n$, since otherwise $X_n(\tau) = Y_n(\tau) = n$ by \ref{coupOddNoneAtAux}.)
\end{itemize}
By the argument of the second item, we can also conclude that, if $X_n(t) \neq Y_n(t)$, then $X_n(\tau), Y_n(\tau)$ were assigned via \ref{coupOddNoneAtAux} for each $\tau \leq t$. Thus, at all such $\tau$, the die in \ref{coupOddNoneAtAux} must have landed $1$ or $2$ (else, $X_n(t) \neq Y_n(t)$ is violated); this occurs with probability $1 - \frac{1}{n+1}$. Hence, by independence,
\begin{equation} \label{eqCoupOddEasyCase}
\P_{ij} ( X_n(t) \neq Y_n(t) ) \leq \left( 1 - \frac{1}{n+1} \right)^t \leq \exp \left( - \frac{t}{n+1} \right)\ \forall\ i,j \in [n] \setminus \{n\}\ s.t.\ i \neq j .
\end{equation}
We next consider the case $X_n(0) = n$ or $Y_n(0) = n$; without loss of generality, assume $X_n(0) = n, Y_n(0) = j \neq n$. Let $\tau \leq t$ and $E_{\tau} = \{ X_n(\tau) \neq n, Y_n(\tau) \neq n \}$. Then
\begin{align}
& \P_{nj} ( X_n(t) \neq Y_n(t) ) \\
& \quad = \P ( X_n(t) \neq Y_n(t) , E_{\tau} | X_n(0) = n, Y_n(0) = j ) \\
& \quad\quad\quad + \P ( X_n(t) \neq Y_n(t) , E_{\tau}^C | X_n(0) = n, Y_n(0) = j ) \\
& \quad = \P ( X_n(t) \neq Y_n(t) | E_{\tau} ) \P ( E_{\tau} | X_n(0) = n, Y_n(0) = j ) \\
& \quad\quad\quad + \P ( X_n(t) \neq Y_n(t) , E_{\tau}^C | X_n(0) = n, Y_n(0) = j ) \\
& \quad \leq \P ( X_n(t) \neq Y_n(t) | E_{\tau} ) + \P (  X_n(t) \neq Y_n(t) , E_{\tau}^C | X_n(0) = n, Y_n(0) = j ) \label{eqCoupOddHardCaseSummands} ,
\end{align}
where the equalities use total probability and the Markov property, and the inequality is immediate. Now for the first summand in \eqref{eqCoupOddHardCaseSummands}, we can use time invariance and \eqref{eqCoupOddEasyCase} to obtain
\begin{equation} \label{eqCoupOddHardCaseSummand1}
\P ( X_n(t) \neq Y_n(t) | E_{\tau} ) = \P ( X_n(t-\tau) \neq Y_n(t-\tau) | E_{0} ) \leq \exp \left( - \frac{t-\tau}{n+1} \right) .
\end{equation}
For the second summand in \eqref{eqCoupOddHardCaseSummands}, we again use $X_n(t) \neq Y_n(t) \Rightarrow X_n(\tau) \neq Y_n(\tau)$ to obtain
\begin{align} \label{eqCoupOddHardCaseSummand2a}
& \P (  X_n(t) \neq Y_n(t) , E_{\tau}^C | X_n(0) = n, Y_n(0) = j )  \\
& \quad \leq \P ( X_n(\tau) \neq Y_n(\tau) , E_{\tau}^C  | X_n(0) = n, Y_n(0) = j )
\end{align}
We next claim (and will return to prove) that
\begin{equation} \label{eqCoupOddFinalClaim}
\{ X_n(\tau) \neq Y_n(\tau) , E_{\tau}^C \} | \{ X_n(0) = n, Y_n(0) = j \} \Rightarrow X_n(\tau') = n\ \forall\ \tau' \leq \tau ,
\end{equation}
i.e.\ conditioned on the event $\{ X_n(0) = n, Y_n(0) = j \}$, the event $\{ X_n(\tau) \neq Y_n(\tau) , E_{\tau}^C \}$ can only occur if the $X_n$-chain is lazy at every step up to $\tau$. In other words, we require the $\tau$ independent coin tosses at the first $\tau$ iterations of \ref{coupOddOneAtAux} to all land tails. Hence, we conclude
\begin{equation} \label{eqCoupOddHardCaseSummand2b}
\P ( X_n(\tau) \neq Y_n(\tau) , E_{\tau}^C  | X_n(0) = n, Y_n(0) = j ) \leq 2^{-\tau} .
\end{equation}
Combining \eqref{eqCoupOddEasyCase}, \eqref{eqCoupOddHardCaseSummands}, \eqref{eqCoupOddHardCaseSummand1}, \eqref{eqCoupOddHardCaseSummand2a}, and \eqref{eqCoupOddHardCaseSummand2b}, we have ultimately shown that for $n$ odd,
\begin{equation} \label{eqCoupOddAlmostDone} 
\max_{i,j \in [n]} \P_{ij} ( X_n(t) \neq Y_n(t) ) \leq \max \left\{ \exp \left( - \frac{t}{n+1} \right), 2^{-\tau} + \exp \left( - \frac{t-\tau}{n+1} \right) \right\} .
\end{equation}
Therefore, if we choose (for example)
\begin{equation} 
\tau \geq \log_2 \left( \frac{2}{\epsilon} \right) , \quad t \geq \tau + (n+1) \log \left( \frac{2}{\epsilon} \right) \geq (n+1) \log \left( \frac{2}{\epsilon} \right) +  \log_2 \left( \frac{2}{\epsilon} \right) ,
\end{equation}
we conclude \eqref{eqCoupOddAlmostDone} is further bounded by $\epsilon$. We thus obtain for some $b_{\epsilon}$ independent of $n$,
\begin{equation} \label{eqExCouplingNoddFinal}
t_{\mix}^{(n)}(\epsilon) \leq (n+1) \log \left( \frac{2}{\epsilon} \right) +  \log_2 \left( \frac{2}{\epsilon} \right) \leq b_{\epsilon} n\ \forall\ n \in \{1,3,\ldots\} .
\end{equation}
Finally, we can combine \eqref{eqExCouplingNevenFinal} and \eqref{eqExCouplingNoddFinal} to obtain for some $a_{\epsilon}, b_{\epsilon}$ independent of $n$,
\begin{equation}
t_{\mix}^{(n)}(\epsilon) \leq \max \{ a_{\epsilon}, b_{\epsilon} \} n\ \forall\ n \geq 6 \quad \Rightarrow \quad t_{\mix}^{(n)}(\epsilon) = O(n) .
\end{equation}
We have completed the proof of $t_{\mix}^{(n)}(\epsilon) = O(n)$, assuming \eqref{eqCoupOddFinalClaim} holds. We now return to prove \eqref{eqCoupOddFinalClaim}. Assume (for the sake of contradiction) that $X_n(\tau^*) = n, X_n(\tau^*+1) \neq n$ for some $\tau^* < \tau$. (i.e.\ the $X_n$-chain was non-lazy at some $\tau^* < \tau$). Then, by \ref{coupOddOneAtAux}, the $Y_n$-chain was lazy at time $\tau^*$, i.e.\ $Y_n(\tau^*) = Y_n(\tau^*+1)$. Now consider two cases:
\begin{enumerate}
\item $\tau^* = \tau-1$: By assumption, $X_n(\tau) = X_n(\tau^*+1) \neq n$. Furthermore, we must have $Y_n(\tau) \neq n$: if instead $Y_n(\tau) = n$, then $n = Y_n(\tau) = Y_n(\tau^*+1) = Y_n(\tau^*)$ (since $\tau^* = \tau-1$ and the $Y_n$-chain was lazy at $\tau^*$), which implies $X_n(\tau^*) = Y_n(\tau^*) = n$, which contradicts $X_n(\tau) \neq Y_n(\tau)$ in \eqref{eqCoupOddFinalClaim}. Hence, we must have $X_n(\tau) \neq n, Y_n(\tau) \neq n$. But this contradicts $E_{\tau}^C$ in \eqref{eqCoupOddFinalClaim}.
\item $\tau^* < \tau-1$: By a similar argument, we have $X_n(\tau^*+1) \neq n, Y_n(\tau^*+1) \neq n$ and $X_n(\tau^*+1) \neq Y_n(\tau^*+1)$. This implies $X_n(\tau^*+2), Y_n(\tau^*+2)$ were assigned via \ref{coupOddNoneAtAux}. In \ref{coupOddNoneAtAux}, the chains only move to $n$ if they move to $n$ together, after which point they remain together forever. Thus, neither chain can move to $n$ at time $\tau^*+2$, else $X_n(\tau) \neq Y_n(\tau)$ in \eqref{eqCoupOddFinalClaim} is contradicted. Repeating this argument for $\tau^*+3,\ldots,\tau$ then contradicts $E_{\tau}^C$ in \eqref{eqCoupOddFinalClaim}.
\end{enumerate}
Since both cases yield contradictions, \eqref{eqCoupOddFinalClaim} is proven.

\begin{appendix}

\section{Basic results} \label{appExisting}

First, we recall some basic properties of total variation distance.
\begin{lemma} \label{lemBasicTV}
Let $\mu, \nu, \eta \in \Delta_{n-1}$. Then the following hold:
\begin{itemize}
\item ($l_1$ equivalence) $\| \mu - \nu \| = \frac{1}{2} \sum_{i=1}^n | \mu(i) - \nu(i) | = \frac{1}{2} \| \mu - \nu \|_1$.
\item (Triangle inequality) $\| \mu - \nu \| \leq \| \mu - \eta \| + \| \eta - \nu \|$.
\item (Convexity) $\| ( \gamma \mu + (1-\gamma) \nu )  - \eta \| \leq \gamma \| \mu - \eta \| + (1-\gamma) \| \nu - \eta \|\ \forall\ \gamma \in (0,1)$.
\item (Coupling) $\| \mu - \nu \| \leq \P ( X \neq Y )$ for any coupling $(X,Y)$ of $(\mu,\nu)$, i.e.\ for any pair of random variables $X$ and $Y$ with respective marginal distributions $\mu$ and $\nu$. Moreover, there exists a coupling $(X,Y)$ of $(\mu,\nu)$ such that $\| \mu - \nu \| = \P ( X \neq Y )$.
\end{itemize}
\end{lemma}
\begin{proof}
For $l_1$ equivalence, see Proposition 4.2 in \cite{levin2009markov}. The triangle inequality and convexity can then be proven using the corresponding $l_1$ properties. For coupling, see Proposition 4.7 in \cite{levin2009markov}.
\end{proof}

We next collect some basic mixing time results. 
\begin{lemma} \label{lemBasicMixing}
Let $P_n \in \mathcal{E}_n\ \forall\ n \in \N$, and let $\epsilon \in (0,1)$ be independent of $n$.
\begin{itemize}
\item For any $n, t \in \N$, $d_n(t) \leq \max_{i,j \in [n]} \| e_i P_n^t - e_j P_n^t \|$.
\item If each $P_n$ is lazy, then $\sup_{\delta \in (0,1)} \liminf_{n \rightarrow \infty} t_{\mix}^{(n)}(\delta) = \infty$.
\item For any $n, k \in \N$, $d_n( k t_{\mix}^{(n)}(\epsilon) ) \leq (2 \epsilon)^k$. (Note this inequality motivates the convential choice $\epsilon = 1/4$, since we then obtain the convenient inequality $d_n( k t_{\mix}^{(n)}) \leq 2^{-k}$.)
\item If $P_n$ is reversible, then $t_{\mix}^{(n)}(\epsilon) \geq ( t_{\rel}^{(n)} - 1 ) \log ( 1 / ( 2 \epsilon ) )$. 
\item If $\{ P_n \}_{n \in \N}$ exhibits pre-cutoff and each $P_n$ is reversible, then $t_{\rel}^{(n)} = o ( t_{\mix}^{(n)}(\epsilon) )$.
\end{itemize}
\end{lemma}
\begin{proof}
The first statement holds by global balance and convexity (see Lemma \ref{lemBasicTV}), i.e.\
\begin{equation}
d_n(t) \leq \max_{i \in [n]} \sum_{j \in [n]} \pi_n(j) \| e_i P_n^t - e_j P_n^t \| \leq \max_{i,j \in [n]} \| e_i P_n^t - e_j P_n^t \| .
\end{equation}
For the second statement, let $i_n \in [n]$ be s.t.\ $\pi_n(i_n) \leq 1/n\ \forall\ n \in \N$ (clearly, such $i_n$ exists).  Then by definition of $d_n(t)$, definition of total variation, and laziness, we have $\forall\ n,t \in \N$,
\begin{equation}
d_n( t ) \geq \| e_{i_n} P_n^t - \pi_n \| \geq ( e_{i_n} P_n^t ) (i_n)  - \pi_n(i_n) \geq 2^{-t} - 1/n .
\end{equation}
As a consequence of this inequality, we obtain
\begin{equation}
t_{\mix}^{(n)}(\delta) \geq \log_2 \left( \frac{1}{\delta+1/n} \right)\ \forall\ n \in \N, \delta \in (0,1) \quad \Rightarrow \quad \sup_{\delta \in (0,1)} \liminf_{n \rightarrow \infty} t_{\mix}^{(n)}(\delta) = \infty .
\end{equation}
For the other statements, see Equation 4.34, Theorem 12.4, and Proposition 18.4 in \cite{levin2009markov}.
\end{proof}

Finally, we have a formula for the stationary distribution of restart perturbations.
\begin{lemma} \label{lemPropRestart}
Let $P_n \in \mathcal{E}_n, \alpha_n \in (0,1), \sigma_n \in \Delta_{n-1}$. Then
\begin{equation}
\pi_{\alpha_n,\sigma_n} = \alpha_n \sum_{t=0}^{\infty} (1-\alpha_n)^t \sigma_n P_n^t ,
\end{equation}
where (we recall) $\pi_{\alpha_n,\sigma_n}$ is the stationary distribution of $P_{\alpha_n,\sigma_n} = (1-\alpha_n) P_n + \alpha_n 1_n^{\trans} \sigma_n$.
\end{lemma}
\begin{proof}
Using global balance and the fact that $\pi_{\alpha_n,\sigma_n}$ sums to 1 (i.e.\ $\pi_{\alpha_n,\sigma_n} 1_n^{\trans} = 1$),
\begin{equation}
\pi_{\alpha_n,\sigma_n} = \pi_{\alpha_n,\sigma_n} P_{\alpha_n,\sigma_n} = (1-\alpha_n) \pi_{\alpha_n,\sigma_n} P_n + \alpha_n \sigma_n \Rightarrow \alpha_n \sigma_n = \pi_{\alpha_n,\sigma_n} ( I_n - (1-\alpha_n) P_n ) .
\end{equation}
Finally, using the well-known identity $(I-A)^{-1} = \sum_{t=0}^{\infty} A^t$ for a matrix $A$ with spectral radius $\rho(A) < 1$, we obtain the desired result.
\end{proof}

\section{Additional proofs from Section \ref{secProofTrichotomy}} 

We next establish the two corollaries whose proofs were deferred from Section \ref{secProofTrichotomy}.

\subsection{Proof of Corollary \ref{corFinalUpper}} \label{appProofFinalUpper}

We first assume $\epsilon = 1/4$; we will then extend the proof to the case $\epsilon \neq 1/4$. In the case $\epsilon = 1/4$ (in fact, any $\epsilon < 1/2$), \eqref{eqCorZeroRegime} follows immediately from Corollary \ref{corUpperWeakAss}. To prove \eqref{eqCorMiddleRegimeUpper}, assume for the sake of contradiction  $\exists\ \{ \tilde{P}_n \}_{n \in \N}$ with $\tilde{P}_n \in B(P_n,\alpha_n)\ \forall\ n \in \N$ and $
\limsup_{n \rightarrow \infty} \| \pi_n - \tilde{\pi}_n \| > 1 - e^{-c}$. Clearly, if this inequality holds, then the interval
\begin{equation}
\left( 0 , \min \left\{ 1/4 ,  \frac{\limsup_{n \rightarrow \infty} \| \pi_n - \tilde{\pi}_n \| - (1-e^{-c})}{e^{-c}}  \right\} \right) 
\end{equation}
is nonempty, so we can choose $\delta$ in this interval. Since $\delta < 1/4$ by construction, \eqref{eqMixMonotone} implies
\begin{equation} \label{eqCorFourthToEps1}
\alpha_n t_{\mix}^{(n)} \leq \alpha_n t_{\mix}^{(n)}(\delta) = \alpha_n t_{\mix}^{(n)}(1-\delta) \frac{t_{\mix}^{(n)}(\delta) }{t_{\mix}^{(n)}(1-\delta) } \leq \alpha_n t_{\mix}^{(n)}  \frac{t_{\mix}^{(n)}(\delta) }{t_{\mix}^{(n)}(1-\delta) }  .
\end{equation}
Hence, using the definition of $c$ and the cutoff assumption,
\begin{equation} \label{eqCorFourthToEps2}
c = \lim_{n \rightarrow \infty} \alpha_n t_{\mix}^{(n)} \leq \lim_{n \rightarrow \infty} \alpha_n t_{\mix}^{(n)}(\delta) \leq \lim_{n \rightarrow \infty} \alpha_n t_{\mix}^{(n)} \times \lim_{n \rightarrow \infty} \frac{t_{\mix}^{(n)}(\delta) }{t_{\mix}^{(n)}(1-\delta) } = c \times 1 = c ,
\end{equation}
so that $\lim_{n \rightarrow \infty} \alpha_n t_{\mix}^{(n)}(\delta) = c$. Assuming for the moment that $t_{\mix}^{(n)}(\delta) \rightarrow \infty$, we can then use \eqref{eqMiddleRegimeUpper} and the choice of $\delta$ to obtain
\begin{equation}
\limsup_{n \rightarrow \infty} \| \pi_n - \tilde{\pi}_n \| \leq 1 - e^{-c} + \delta e^{-c} < \limsup_{n \rightarrow \infty} \| \pi_n - \tilde{\pi}_n \| ,
\end{equation}
which is a contradiction. Now to see why $t_{\mix}^{(n)}(\delta) \rightarrow \infty$ holds, first note that $\forall\ \delta' \in (0,\delta)$, 
\begin{equation}
t_{\mix}^{(n)}(\delta) = t_{\mix}^{(n)}(\delta') \frac{t_{\mix}^{(n)}(\delta)}{t_{\mix}^{(n)}(\delta')} \geq t_{\mix}^{(n)}(\delta') \frac{t_{\mix}^{(n)}(1-\delta')}{t_{\mix}^{(n)}(\delta')} ,
\end{equation}
where the inequality holds by \eqref{eqMixMonotone}. Hence, by cutoff, we obtain $\forall\ \delta' \in (0,\delta)$,
\begin{equation}
\liminf_{n \rightarrow \infty} t_{\mix}^{(n)}(\delta) \geq \liminf_{n \rightarrow \infty} t_{\mix}^{(n)}(\delta') .
\end{equation}
On the other hand, the previous inequality immediately holds $\forall\ \delta' \in [\delta,1)$ by \eqref{eqMixMonotone}. Therefore,
\begin{equation}
\liminf_{n \rightarrow \infty} t_{\mix}^{(n)}(\delta) \geq \sup_{\delta' \in (0,1)} \liminf_{n \rightarrow \infty} t_{\mix}^{(n)}(\delta') = \infty ,
\end{equation}
where the equality holds by Lemma \ref{lemBasicMixing} in Appendix \ref{appExisting}.

Finally, we extend the upper bounds to $\epsilon \neq 1/4$, for which it suffices to show
\begin{equation} \label{eqProofAllRegExt}
\lim_{n \rightarrow \infty} \alpha_n t_{\mix}^{(n)}(\epsilon) = c \quad \Rightarrow \quad \lim_{n \rightarrow \infty} \alpha_n t_{\mix}^{(n)} = c ,
\end{equation}
after which we can invoke the result from the case $\epsilon = 1/4$ to complete the proof. \eqref{eqProofAllRegExt} is an almost direct consequence of cutoff. To prove it, we first use \eqref{eqMixMonotone} to obtain
\begin{align}
\epsilon \in (0,1/4) & \Rightarrow \alpha_n t_{\mix}^{(n)}(\epsilon) \geq \alpha_n t_{\mix}^{(n)} = \alpha_n t_{\mix}^{(n)}(\epsilon) \frac{ t_{\mix}^{(n)} }{ t_{\mix}^{(n)}(\epsilon) } \geq \alpha_n t_{\mix}^{(n)}(\epsilon) \frac{ t_{\mix}^{(n)}(1-\epsilon) }{ t_{\mix}^{(n)}(\epsilon) } , \\
\epsilon \in (1/4,3/4] & \Rightarrow \alpha_n t_{\mix}^{(n)}(\epsilon) \leq \alpha_n t_{\mix}^{(n)} = \alpha_n t_{\mix}^{(n)}(\epsilon) \frac{ t_{\mix}^{(n)} }{ t_{\mix}^{(n)}(\epsilon) } \leq \alpha_n t_{\mix}^{(n)}(\epsilon) \frac{ t_{\mix}^{(n)} }{ t_{\mix}^{(n)}(3/4) } , \\
\epsilon \in (3/4,1) & \Rightarrow \alpha_n t_{\mix}^{(n)}(\epsilon) \leq \alpha_n t_{\mix}^{(n)} = \alpha_n t_{\mix}^{(n)}(\epsilon) \frac{ t_{\mix}^{(n)} }{ t_{\mix}^{(n)}(\epsilon) } \leq \alpha_n t_{\mix}^{(n)}(\epsilon) \frac{ t_{\mix}^{(n)}(1-\epsilon) }{ t_{\mix}^{(n)}(\epsilon) } .
\end{align}
Now letting $n \rightarrow \infty$ and using cutoff in the three cases, \eqref{eqProofAllRegExt} follows as in \eqref{eqCorFourthToEps2}.

\subsection{Proof of Corollary \ref{corLowerWithoutDeltas}} \label{appProofLowerWithoutDeltas}

We begin with \eqref{lemLowerStongerInf}. Let $\delta_k = 2^{-(k+1)} / 3 \in (0,1/2)\ \forall\ k \in \N$. We claim $\lim_{n \rightarrow \infty} \alpha_n t_{\mix}^{(n)}(\delta_k) = \infty\ \forall\ k \in \N$, which we prove as follows:
\begin{itemize}
\item If $\delta_k \leq \epsilon$, then $t_{\mix}^{(n)}(\delta_k) \geq t_{\mix}^{(n)}(\epsilon)$ by \eqref{eqMixMonotone}, so $\alpha_n t_{\mix}^{(n)}(\delta_k) \rightarrow \infty$ by assumption.
\item If $\delta_k > \epsilon$ and $\epsilon < 1/2$, then $\delta_k < 1/2 < 1 - \epsilon$, so $t_{\mix}^{(n)}(\delta_k) \geq t_{\mix}^{(n)}(1-\epsilon)$ by \eqref{eqMixMonotone}, and
\begin{equation}
\alpha_n t_{\mix}^{(n)}(\delta_k) = \alpha_n t_{\mix}^{(n)}(\epsilon) \frac{  t_{\mix}^{(n)}(\delta_k) }{  t_{\mix}^{(n)}(\epsilon) } \geq \alpha_n t_{\mix}^{(n)}(\epsilon) \frac{  t_{\mix}^{(n)}(1-\epsilon) }{  t_{\mix}^{(n)}(\epsilon) } \xrightarrow[n \rightarrow \infty]{} \infty  ,
\end{equation}
where the limit holds since $\alpha_n t_{\mix}^{(n)}(\epsilon) \rightarrow \infty$ by assumption and since $t_{\mix}^{(n)}(1-\epsilon) / t_{\mix}^{(n)}(\epsilon)$ is lower bounded by a positive constant as $n \rightarrow \infty$ by pre-cutoff.
\item The final case, $\delta_k > \epsilon$ and $\epsilon \geq 1/2$, cannot occur, since $\delta_k < 1/2\ \forall\ k \in \N$.
\end{itemize}
We have verified the assumptions of Corollary \ref{corLowerWithDeltas}, so for each $k \in \N$ we can find $\{ \tilde{P}_n^{(k)} \}_{n \in \N}$ s.t.\ $\tilde{P}_n^{(k)} \in  B(P_n,\alpha_n)\ \forall\ n \in \N$, and, denoting the stationary distribution of $\tilde{P}_n^{(k)}$ by $\tilde{\pi}_n^{(k)}$,
\begin{equation} \label{eqInfInvokeWeak}
\liminf_{n \rightarrow \infty} \| \pi_n - \tilde{\pi}_n^{(k)} \| \geq 1 - 3 \delta_k = 1 - 2^{-(k+1)} .
\end{equation}
Note that, as a consequence of \eqref{eqInfInvokeWeak}, $\forall\ k \in \N\ \exists\ N_k \in \N$ s.t.\
\begin{equation} \label{eqNkDefn}
\| \pi_n - \tilde{\pi}_n^{(k)} \| > 1 - 2^{-k}\ \forall\ n \geq N_k .
\end{equation}
We complete the proof of \eqref{lemLowerStongerInf} by considering two cases. 
\begin{enumerate}
\item $\lim_{k \rightarrow \infty} N_k = \infty$. Our goal is to use $\{ \tilde{P}_n^{(k)} \}_{n , k \in \N}$ to construct $\{ \tilde{P}_n \}_{n \in \N}$ satisfying \eqref{lemLowerStongerInf}. The construction proceeds as follows:
\begin{itemize}
\item If $n < \min_{k \in \N} N_k$, set $\tilde{P}_n = \tilde{P}_n^{(1)}$. (The choice $k=1$ is arbitrary.)
\item If $n \geq \min_{k \in \N} N_k$, let $k_n = \max \{ k \in \N : n \geq N_k \}$ and set $\tilde{P}_n = \tilde{P}_n^{(k_n)}$. (Note $n \geq \min_{k \in \N} N_k$ guarantees $\{ k \in \N : n \geq N_k \} \neq \emptyset$, while $N_k \rightarrow \infty$ guarantees $| \{ k \in \N : n \geq N_k \} | < \infty$, so $k_n$ is well-defined.)
\end{itemize}
Note that, since $\tilde{P}_n^{(k)} \in  B(P_n,\alpha_n)\ \forall\ n,k \in \N$ by Corollary \ref{corLowerWithDeltas}, this construction guarantees $\tilde{P}_n \in  B(P_n,\alpha_n)\ \forall\ n \in \N$ as well. Additionally, for $n \geq \min_{k \in \N} N_k$, we have $n \geq N_{k_n}$ by definition. Hence, because $\tilde{\pi}_n = \tilde{\pi}_n^{(k_n)}$ for all such $n$, we can use \eqref{eqNkDefn} to obtain
\begin{gather} \label{eqKnToInfSuff}
\| \pi_n - \tilde{\pi}_n \| = \| \pi_n - \tilde{\pi}_n^{(k_n)} \| > 1 - 2^{-k_n}\ \forall\  n \geq \min_{k \in \N} N_k  \\
\Rightarrow  \liminf_{n \rightarrow \infty} \| \pi_n - \tilde{\pi}_n \|  \geq 1 - \limsup_{n \rightarrow \infty} 2^{-k_n} .
\end{gather}
Thus, to complete the proof, it suffices to show $k_n \rightarrow \infty$ as $n \rightarrow \infty$. For this, let $M > 0$ and define $N^{(m)} = \max \{ N_1, \ldots, N_{\ceil{M}} \}$. Then $k_n \geq \ceil{M} \geq M\ \forall\ n \geq N^{(m)}$, so since $M > 0$ was arbitrary, $k_n \rightarrow \infty$ as $n \rightarrow \infty$ follows. 
\item $\lim_{k \rightarrow \infty} N_k < \infty$. Here the construction is much simpler: we set $\tilde{P}_n = \tilde{P}_n^{(n)}\ \forall\ n \in \N$. Then for all $n$ sufficiently large, $n \geq N_n$, so for such $n$,
\begin{equation}
\| \pi_n - \tilde{\pi}_n \| = \| \pi_n - \tilde{\pi}_n^{(n)} \| > 1 - 2^{-n} ,
\end{equation}
from which it is clear that $\lim_{n \rightarrow \infty} \| \pi_n - \tilde{\pi}_n \|  = 1$.
\end{enumerate}

We prove \eqref{lemLowerStongerC} similarly. Similar to above, let $\delta_k = 2^{-(k+1)}/3 \in (0,1/2)$; we claim that $\lim_{n \rightarrow \infty} \alpha_n t_{\mix}^{(n)}(\delta_k) = c$. To prove this, first note (provided the limits exist in $(0,\infty)$)
\begin{equation}
\lim_{n \rightarrow \infty} \alpha_n t_{\mix}^{(n)}(\delta_k) = \lim_{n \rightarrow \infty} \alpha_n t_{\mix}^{(n)}(\epsilon) \lim_{n \rightarrow \infty} \frac{ t_{\mix}^{(n)}(\delta_k) }{  t_{\mix}^{(n)}(\epsilon) } = c \lim_{n \rightarrow \infty} \frac{ t_{\mix}^{(n)}(\delta_k) }{  t_{\mix}^{(n)}(\epsilon) } ,
\end{equation}
so it suffices to show $\lim_{n \rightarrow \infty}  t_{\mix}^{(n)}(\delta_k) / t_{\mix}^{(n)}(\epsilon)  = 1$. This can be proven as follows:
\begin{itemize}
\item If $\epsilon \leq \delta_k \leq 1-\epsilon$, we have $t_{\mix}^{(n)}(1-\epsilon) \leq t_{\mix}^{(n)}(\delta_k) \leq t_{\mix}^{(n)}(\epsilon)$ by \eqref{eqMixMonotone}, so by cutoff,
\begin{equation}
1 = \lim_{n \rightarrow \infty} \frac{ t_{\mix}^{(n)}(1-\epsilon) }{  t_{\mix}^{(n)}(\epsilon) } \leq  \lim_{n \rightarrow \infty} \frac{ t_{\mix}^{(n)}(\delta_k) }{  t_{\mix}^{(n)}(\epsilon) } \leq \lim_{n \rightarrow \infty} \frac{ t_{\mix}^{(n)}(\epsilon) }{  t_{\mix}^{(n)}(\epsilon) } = 1 .
\end{equation}
\item If $\delta_k \leq \epsilon \leq 1-\epsilon$, we have $t_{\mix}^{(n)}(1-\delta_k) \leq t_{\mix}^{(n)}(\epsilon) \leq t_{\mix}^{(n)}(\delta_k)$ by \eqref{eqMixMonotone}, so by cutoff,
\begin{equation}
1 = \lim_{n \rightarrow \infty} \frac{ t_{\mix}^{(n)}(1-\delta_k) }{  t_{\mix}^{(n)}(\delta_k) } \leq \lim_{n \rightarrow \infty} \frac{ t_{\mix}^{(n)}(\epsilon) }{  t_{\mix}^{(n)}(\delta_k) } \leq \lim_{n \rightarrow \infty} \frac{ t_{\mix}^{(n)}(\epsilon) }{  t_{\mix}^{(n)}(\epsilon) } = 1 . 
\end{equation}
\item If $1-\epsilon \leq \delta_k \leq \epsilon$ or $\delta_k \leq 1-\epsilon \leq \epsilon$, the result holds by reversing the roles of $\epsilon$ and $1-\epsilon$.
\item Finally, $\epsilon \leq 1-\epsilon \leq \delta_k$ and $1-\epsilon \leq \epsilon \leq \delta_k$ cannot occur since $\delta_k < 1/2$.
\end{itemize}
We have shown $\delta_k \in (0,1/2)$ and $\lim_{n \rightarrow \infty} \alpha_n t_{\mix}^{(n)}(\delta_k) = c$ $\forall\ k \in \N$. Hence, for each $k$, we can use Corollary \ref{corLowerWithDeltas} to find $\{ \tilde{P}_n^{(k)} \}_{n \in \N}$ s.t.\ $\tilde{P}_n^{(k)} \in  B(P_n,\alpha_n)\ \forall\ n \in \N$, and
\begin{equation}
\liminf_{n \rightarrow \infty} \| \pi_n - \tilde{\pi}_n^{(k)} \| \geq  1 - e^{-c} - 2^{-(k+1)} .
\end{equation}
From here, the proof can be completed in a similar manner as the proof of \eqref{lemLowerStongerInf}, by simply replacing $1$ with $1-e^{-c}$ in the analysis following \eqref{eqInfInvokeWeak}.

\section{Proofs for example chains}

Finally, we prove Propositions \ref{propExamplesMixing} and \ref{propExamplesPerturb} from Section \ref{secExamples}, which bounds the mixing times and perturbation error, respectively, for the example chains.

\subsection{Proof of Proposition \ref{propExamplesMixing}} \label{appProofExamplesMixing}

\subsubsection{Winning streak reversal}

Most of the arguments are recounted from Section 4.6 of \cite{levin2009markov}. First, for $i \in [n-1]$, the chain started from $i$ reaches stationarity in $i$ steps, i.e.\
\begin{equation}
e_i P_n^i = e_i P_n^{i-1} P_n = e_1 P_n = \pi_n .
\end{equation}
It remains to analyze the chain starting from $n$. First, we claim that for $j \in[n-1]$,
\begin{equation} \label{eqWsrFromNclaim}
e_n P_n^j = \sum_{i=1}^j 2^{i-j-1} e_{n-i} + 2^{-j} e_n .
\end{equation}
This claim can be proven inductively: for $j=1$, the left side of \eqref{eqWsrFromNclaim} is $(e_{n-1}+e_n)/2$ by \eqref{eqWsrTransAndStat}, while the right side of \eqref{eqWsrFromNclaim} is clearly $(e_{n-1}+e_n)/2$; assuming true for $j$, we have
\begin{align}
e_n P_n^{j+1} & = e_n P_n P_n^j = \frac{1}{2} ( e_{n-1} + e_n ) P_n^j  = \frac{1}{2} e_{n-1-j} + \frac{1}{2} \left( \sum_{i=1}^j 2^{i-j-1} e_{n-i} + 2^{-j} e_n \right) \\
& = 2^{-1} e_{n-(j+1)} + \sum_{i=1}^j 2^{i-(j+1)-1} e_{n-i} + 2^{-(j+1)} e_n \\
& = \sum_{i=1}^{j+1} 2^{i-(j+1)-1} e_{n-i} + 2^{-(j+1)} e_n ,
\end{align}
which establishes \eqref{eqWsrFromNclaim}. Now taking $j = n-1$ in \eqref{eqWsrFromNclaim}, we obtain
\begin{equation}
e_n P_n^{n-1} = \sum_{i=1}^{n-1} 2^{i-(n-1)-1} e_{n-i} + 2^{-(n-1)} e_n = 2^{-1} e_1 + \cdots + 2^{-(n-1)} e_{n-1} + 2^{-(n-1)} e_n = \pi_n .
\end{equation}
To summarize, we have shown $e_i P_n^i = \pi_n\ \forall\ i \in [n-1]$ and $e_i P_n^{n-1} = \pi_n$, which implies
\begin{equation}
d_n(n-1) = \max_{i \in [n]} \| e_i P_n^{n-1} - \pi_n \| = 0 \quad \Rightarrow \quad t_{\mix}^{(n)}(1-\epsilon) , t_{\mix}^{(n)}(\epsilon) \leq n-1 .
\end{equation}
For a lower bound on the $\epsilon$-mixing time, note that, by \eqref{eqWsrFromNclaim}, $P_n^{n-2}(n,1) = 0$, where $P_n^{n-2}(n,1)$ is the $(n,1)$-th element of $P_n^{n-2}$. Hence, we immediately obtain
\begin{equation}
d_n(n-2) \geq \| e_n P_n^{n-2} - \pi_n \| \geq \pi_n(1) - P_n^{n-2}(n,1) = \frac{1}{2} > \epsilon \quad \Rightarrow \quad t_{\mix}^{(n)}(\epsilon) > n-2 ,
\end{equation}
so, combining with the above, we conclude $t_{\mix}^{(n)}(\epsilon) = n-1$. Finally, to lower bound the $(1-\epsilon)$-mixing time, first note that for any $t \in \{0,\ldots,n-2\}$, we have $e_{n-1} P_n^t = e_{n-1-t}$, so
\begin{equation}
d_n(t) \geq \| e_{n-1} P_n^t - \pi_n \| = \| e_{n-1-t} - \pi_n \| \geq 1 - \pi_n(n-1-t) = 1 - 2^{-n+1+t} .
\end{equation}
Hence, for $t < n -1 - \log_2(1/\epsilon)$, we obtain
\begin{equation}
d_n(t) \geq 1 - 2^{-n+1+t} > 1 - 2^{-\log_2(1/\epsilon)} = 1 - \epsilon \quad \Rightarrow \quad t_{\mix}^{(n)}(1-\epsilon) \geq n-1 -\log_2(1/\epsilon) .
\end{equation}

\subsubsection{Complete graph bijection}

We first show $t_{\mix}^{(n)}(1-\epsilon) = 1$ for $n$ large. For $n$ even, we have by Lemma \ref{lemBasicTV} and \eqref{eqCgbTrans}-\eqref{eqCgbStat},
\begin{align}
2 \|  e_1 P_n - \pi_n \| & =  | P_n(1,1) - \pi_n(1) | + \sum_{j \in N(1)} | P_n(1,j) - \pi_n(j) | \\
& \quad\quad + \sum_{j \in [n] \setminus ( \{1\} \cup N(1) ) } | P_n(1,j) - \pi_n(j) |  \\
& =   \left|  \frac{1}{2} - \frac{1}{n}  \right| +  \frac{n}{2} \left| \frac{1}{n} - \frac{1}{n} \right| + \left( \frac{n}{2} - 1 \right) \left| 0 - \frac{1}{n}  \right|  \xrightarrow[n \rightarrow \infty]{} 1 ,
\end{align}
where (we recall) $N(i)$ are the neighbors of $i$ (see \eqref{eqNeighbors} in Section \ref{secProofExampleMain}). Thus, by symmetry, $\max_{i \in [n]} \| e_i P_n - \pi_n \| \rightarrow 1/2$ along even $n$. If $n$ is odd, we similarly have
\begin{align}
2 \|  e_1 P_n - \pi_n \| & = | P_n(1,1) - \pi_n(1) | + | P_n(1,n) - \pi_n(n) |  \\
& \quad\quad + \sum_{j \in N(1) \setminus \{n\} } | P_n(1,j) - \pi_n(j) |  + \sum_{ j \in [n] \setminus ( \{1\} \cup N(1) ) }  | P_n(1,j) - \pi_n(j) | \\
& = \left| \frac{1}{2} - \frac{n+1}{(n+3)(n-1)} \right| + \left| \frac{1}{n+1} - \frac{2}{n+3} \right| \\
& \quad\quad + \frac{n-1}{2} \left| \frac{1}{n+1} - \frac{n+1}{(n+3)(n-1)} \right|  \\
& \quad\quad + \frac{n-3}{2} \left| 0 -  \frac{n+1}{(n+3)(n-1)} \right|  \xrightarrow[n \rightarrow \infty]{} 1 ,
\end{align}
so, by symmetry, $\max_{i \in [n-1]} \| e_i P_n - \pi_n \| \rightarrow 1/2$ along odd $n$. We also note
\begin{align}
2 \|  e_n P_n - \pi_n \| & = \sum_{j=1}^{n-1} | P_n(n,j) - \pi_n(j) | + | P_n(n,n) - \pi_n | \\
& = (n-1) \left| \frac{1}{2(n-1)} - \frac{n+1}{(n+3)(n-1)} \right| + \left| \frac{1}{2} - \frac{2}{n+3} \right| \xrightarrow[n \rightarrow \infty]{} \frac{1}{2} ,
\end{align}
so $\max_{i \in [n]} \| e_i P_n - \pi_n \| \rightarrow 1/2$ along odd $n$. Combined with the analysis for $n$ even,
\begin{equation} \label{eqMixToHalfInOneStep}
\limsup_{n \rightarrow \infty} d_n(1) = \limsup_{n \rightarrow \infty} \max_{i \in [n]} \| e_i P_n - \pi_n \| \leq \frac{1}{2} < 1 - \epsilon , 
\end{equation}
so $t_{\mix}^{(n)}(1-\epsilon) \leq 1$ for large $n$. Finally, since $\| e_i - \pi_n \| \geq 1 - \pi_n(i)\ \forall\ i \in [n]$, $d_n(0) \geq 1 - \min_{i \in [n]} \pi_n(i) \geq 1 - 1/n > 1 - \epsilon$ for fixed $\epsilon$ and $n$ large, so $t_{\mix}^{(n)}(1-\epsilon) > 0$ for such $n$. Taken together, we obtain $t_{\mix}^{(n)}(1-\epsilon) = 1$ for such $n$.

Finally, we show $t_{\mix}^{(n)}(\epsilon) = \Theta(n)$ for the CGB. We have already proven the upper bound in Section \ref{secProofExampleMain}, so we only need to show $t_{\mix}^{(n)}(\epsilon) = \Omega(n)$. For $n$ even, the intuition is that the stationary distribution places equal weight on both cliques, whereas (for small $t$) the distribution of $X_n(t)$ is  biased towards $[n/2]$ if $X_n(0) = 1$. Hence, we write
\begin{equation} \label{eqExLowerBoundApproach}
d_n(t) \geq \| e_1 P_n^t - \pi_n \| \geq P_n^t( 1, [n/2] ) - \pi_n( [n/2] ) = P_n^t( 1, [n/2] ) - \frac{1}{2} ,
\end{equation}
where $P_n^t(i,j)$ is the $(i,j)$-th element of $P_n^t$ for $i,j \in [n]$ and $P_n^t(i,A) = \sum_{j \in A} P_n^t(i,j)$ for $A \subset [n]$. It remains to lower bound $P_n^t( 1, [n/2] )$. For this, we claim
\begin{equation} \label{eqExLowerBoundClaim}
P_n^t(i,[n/2]) \geq \left( 1 - \frac{1}{n} \right)^t\ \forall\ t \in \Z_+,  i \in [n/2] .
\end{equation}
We prove \eqref{eqExLowerBoundClaim} by induction. For $t = 0$, \eqref{eqExLowerBoundClaim} is immediate. Assuming \eqref{eqExLowerBoundClaim} holds for $t$, we have
\begin{align}
P_n^{t+1}(i,[n/2]) & = \sum_{k \in [n]} P_n(i,k) P_n^t(k,[n/2]) \geq \sum_{k \in [n/2]} P_n(i,k) P_n^t(k,[n/2]) \\
& \geq \left( 1 - \frac{1}{n} \right)^t P_n(i,[n/2]) = \left( 1 - \frac{1}{n} \right)^{t+1} ,
\end{align}
where the first inequality holds by nonnegativity of $P_n$, the second inequality is the inductive hypothesis, and the final equality holds by \eqref{eqCgbTrans}. This proves \eqref{eqExLowerBoundClaim}. Substituting into \eqref{eqExLowerBoundApproach},
\begin{equation}
d_n(t) \geq \left( 1 - \frac{1}{n} \right)^t - \frac{1}{2} \geq \left( 1 - \frac{t}{n} \right) - \frac{1}{2} = \frac{1}{2} - \frac{t}{n} ,
\end{equation}
where we have also used Bernoulli's inequality. The following is then immediate:
\begin{equation} \label{eqExLowerFinalEven}
t < n \left( \frac{1}{2} - \epsilon \right) \quad \Rightarrow \quad d_n(t) > \epsilon \quad \Rightarrow \quad  t_{\mix}^{(n)}(\epsilon) > n\left( \frac{1}{2} - \epsilon \right) .
\end{equation}
We next assume $n$ is odd. Here the argument is nearly identical: since by \eqref{eqCgbTrans},
\begin{equation}
P_n \left( i , \left[ \frac{n-1}{2} \right] \right) = 1 - \frac{2}{n+1}\ \forall\ i \in \left[ \frac{n-1}{2} \right] ,
\end{equation}
we can use an inductive argument as above to obtain
\begin{equation}
P_n^t \left( 1 , \left[ \frac{n-1}{2} \right] \right) \geq \left( 1 - \frac{2}{n+1} \right)^t\ \forall\ t \in \Z_+ .
\end{equation}
On the other hand, by \eqref{eqCgbTrans} we have
\begin{equation}
\pi_n \left( \left[ \frac{n-1}{2} \right] \right) = \frac{n-1}{2}  \frac{n+1}{(n-1)(n+3)} \leq \frac{1}{2} .
\end{equation}
Hence, combining the previous two lines, and using Bernoulli's inequality, we obtain
\begin{equation}
d_n(t) \geq \| e_1 P_n^t - \pi_n \| \geq \frac{1}{2} - \frac{2t}{n+1} .
\end{equation}
The following implications are then immediate:
\begin{equation} \label{eqExLowerFinalOdd}
t < \frac{n+1}{2} \left( \frac{1}{2} - \epsilon \right) \quad \Rightarrow \quad d_n(t) > \epsilon \quad \Rightarrow \quad t_{\mix}^{(n)}(\epsilon) > \frac{n+1}{2} \left( \frac{1}{2} - \epsilon \right) .
\end{equation}
Combining \eqref{eqExLowerFinalEven} and \eqref{eqExLowerFinalOdd}, we conclude $t_{\mix}^{(n)}(\epsilon) = \Omega(n)$.

\subsection{Proof of Proposition \ref{propExamplesPerturb}} \label{appProofExamplesPerturb}

\subsubsection{Winning streak reversal}

Let $\{ \alpha_n \}_{n \in \N}$, $\{ \sigma_n \}_{n \in \N}$, $c_1$, $c_2$, and $c_3$ be as in the statement of the proposition. For $n \in \N$, set $m_n = \floor{ n^{ c_1 (1 + c_2) / 2} }$. Then by $\alpha_n = \Theta ( n^{-c_1} )$, $c_1 > 0$, and $c_2 > 1$,
\begin{equation} \label{eqWsrAlphanMnToInf}
\alpha_n m_n = \Theta \left( n^{ - c_1 } n^{ c_1 (1 + c_2) / 2 } \right) = \Theta \left( n^{ c_1 ( c_2 - 1 ) / 2 } \right) \quad \Rightarrow \quad \lim_{n \rightarrow \infty} \alpha_n m_n = \infty .
\end{equation}
Again using $\alpha_n = \Theta ( n^{-c_1} )$, $c_1 > 0$, and $c_2 > 1$, we also observe
\begin{equation}
\floor*{ c_3 \alpha_n^{-c_2} } - m_n = \Theta \left( n^{c_1 c_2} - n^{ c_1 (1 + c_2) / 2} \right) \quad \Rightarrow \quad \lim_{n \rightarrow \infty} \left( \floor*{ c_3 \alpha_n^{-c_2} } - m_n \right) = \infty .
\end{equation}
Consequently, we can find a sequence of positive integers $\{ m_n' \}_{n \in \N}$ such that
\begin{equation} \label{eqWsrPerturbMnPrime}
\floor*{ c_3 \alpha_n^{-c_2} } - m_n  + 2 > m_n'\ \forall\ n \in \N \textrm{ sufficiently large} , \quad \lim_{n \rightarrow \infty} m_n' = \infty .
\end{equation}
Now letting $e_{ [m_n'] } = \sum_{i \in [m_n']} e_i = \sum_{i=1}^{m_n'} e_i$, we can use Lemma \ref{lemPropRestart} to obtain
\begin{align} 
& \pi_{\alpha_n,\sigma_n} ( [m_n'] ) = \alpha_n \sum_{t=0}^{\infty} (1-\alpha_n)^t \sigma_n P_n^t e_{ [m_n'] }^{\trans} = \alpha_n \sum_{t=0}^{\infty} (1-\alpha_n)^t \sum_{i=1}^n \sigma_n(i) e_i P_n^t e_{ [m_n'] }^{\trans} \\
& \quad\quad= \alpha_n \sum_{t=0}^{m_n-1} (1-\alpha_n)^t \sum_{i=1}^{ \floor{ c_3 \alpha_n^{-c_2} } } \sigma_n(i) e_i P_n^t e_{ [m_n'] }^{\trans} + \alpha_n \sum_{t=m_n}^{\infty} (1-\alpha_n)^t \sum_{i=1}^n \sigma_n(i) e_i P_n^t e_{ [m_n'] }^{\trans} \label{eqWsrPerturbTwoTerms} \\
& \quad\quad\quad\quad + \alpha_n \sum_{t=0}^{m_n-1} (1-\alpha_n)^t \sum_{i=\floor{ c_3 \alpha_n^{-c_2} }+1}^{ n } \sigma_n(i) e_i P_n^t e_{ [m_n'] }^{\trans} . \label{eqWsrPerturbLastTerm} 
\end{align}
To bound the summands in \eqref{eqWsrPerturbTwoTerms}, we use $e_i P_n^t e_{ [m_n'] }^{\trans} \leq 1\ \forall\ i , t$ by row stochasticity to obtain
\begin{align}
\alpha_n \sum_{t=0}^{m_n-1} (1-\alpha_n)^t \sum_{i=1}^{ \floor{ c_3 \alpha_n^{-c_2} } } \sigma_n(i) e_i P_n^t e_{ [m_n'] }^{\trans} & \leq \alpha_n \sum_{t=0}^{m_n-1} (1-\alpha_n)^t \sum_{i=1}^{ \floor{ c_3 \alpha_n^{-c_2} } } \sigma_n(i) \\
& \leq \sum_{i=1}^{ \floor{ c_3 \alpha_n^{-c_2} } } \sigma_n(i) , \\
\alpha_n \sum_{t=m_n}^{\infty} (1-\alpha_n)^t \sum_{i=1}^n \sigma_n(i) e_i P_n^t e_{ [m_n'] }^{\trans} & \leq \alpha_n \sum_{t=m_n}^{\infty} (1-\alpha_n)^t = (1-\alpha_n)^{m_n} \leq e^{-\alpha_n m_n }.
\end{align}
We next consider \eqref{eqWsrPerturbLastTerm}. First note that, whenever $i-t > m_n' > 0$, we have by \eqref{eqWsrTransAndStat}, $e_i P_n^t e_{ [m_n'] } = e_{i-t} e_{ [m_n'] } = 0$. Also, every $i,t$ pair in the summation in \eqref{eqWsrPerturbLastTerm} satisfies, for $n$ sufficiently large by \eqref{eqWsrPerturbMnPrime},
\begin{equation}
i-t \geq  \floor{ c_3 \alpha_n^{-c_2} } + 1  - ( m_n - 1 ) =  \floor{ c_3 \alpha_n^{-c_2} } - m_n +  2 > m_n' ,
\end{equation}
which implies \eqref{eqWsrPerturbLastTerm} is zero for all $n$ large. We have therefore shown
\begin{equation}
\limsup_{n \rightarrow \infty} \pi_{\alpha_n,\sigma_n} ( [m_n'] ) \leq \limsup_{n \rightarrow \infty}  \left( \sum_{i=1}^{ \floor{ c_3 \alpha_n^{-c_2} } } \sigma_n(i)  + \exp ( -\alpha_n m_n ) \right)  = 0 ,
\end{equation}
where the equality holds by assumption and \eqref{eqWsrAlphanMnToInf}. Since also $\pi_{\alpha_n,\sigma_n} ( [m_n'] ) \geq 0\ \forall\ n \in \N$, we conclude $\lim_{n \rightarrow \infty} \pi_{\alpha_n,\sigma_n} ( [m_n'] ) = 0$. On the other hand,
\begin{equation}
\pi_n ( [m_n'] ) = \sum_{i=1}^{m_n'} \pi_n(i) = \sum_{i=1}^{m_n'} 2^{-i} = 1 - 2^{-m_n'} \xrightarrow[n \rightarrow \infty]{} 1,
\end{equation}
where the limit holds since $m_n' \rightarrow \infty$ by \eqref{eqWsrPerturbMnPrime}. Combining arguments, we obtain
\begin{align}
1 \geq \limsup_{n \rightarrow \infty} \| \pi_n - \pi_{\alpha_n,\sigma_n} \| & \geq \liminf_{n \rightarrow \infty} \| \pi_n - \pi_{\alpha_n,\sigma_n} \| \\
& \geq \liminf_{n \rightarrow \infty} \left( \pi_n ( [m_n'] ) - \pi_{\alpha_n,\sigma_n} ( [m_n'] ) \right) = 1 .
\end{align}

\subsubsection{Complete graph bijection}

Let $\{ \alpha_n \}_{n \in \N}, \{ \tilde{P}_n \}_{n \in \N}$ be given. Then
\begin{align}
\| \pi_n - \tilde{\pi}_n \| 
& \leq \max_{i \in [n]} \| \pi_n - e_i P_n \|  + \max_{i \in [n]} \| e_i P_n - e_i \tilde{P}_n \| \leq d_n(1) + \alpha_n\ \forall\ n \in \N ,
\end{align}
where we have used Lemma \ref{lemBasicTV}, global balance, and the fact that $\tilde{P}_n \in  B(P_n,\alpha_n)$. Thus, using \eqref{eqMixToHalfInOneStep} from Appendix \ref{appProofExamplesMixing} and the assumption $\limsup \alpha_n = \bar{\alpha}$, we obtain
\begin{equation}
\limsup_{n \rightarrow \infty} \| \pi_n - \tilde{\pi}_n \| \leq \limsup_{n \rightarrow \infty} d_n(1) + \limsup_{n \rightarrow \infty} \alpha_n = \frac{1}{2} + \bar{\alpha} ,
\end{equation}
which completes the proof.

\end{appendix}

\section*{Acknowledgements}

This work was supported by NSF award CCF 2008130. Most of the work was completed while the first author was at the University of Michigan.

\bibliographystyle{imsart-number} 
\bibliography{references}  

\end{document}